\documentclass[12pt]{amsart}
\usepackage{graphicx}
\usepackage{commath}
\usepackage{stmaryrd}
\usepackage{xcolor}
\usepackage{bm}
\usepackage{hyperref}
\usepackage{mathabx}
\usepackage{tikz}
\usepackage{enumitem}

\usetikzlibrary{cd}
\setlist[enumerate]{label=\arabic*., ref=\arabic*}

\newtheorem{lemma}{Lemma}
\newtheorem{corollary}{Corollary}
\newtheorem{theorem}{Theorem}
\theoremstyle{definition}
\newtheorem{definition}{Definition}
\newtheorem{remark}{Remark}

\newtheorem{proposition}{Proposition}
\newtheorem*{notation}{Notation}
\newtheorem*{connecting}{Connecting Principle}

\newcommand{\PPi}{\boldsymbol{\Pi}}
\newcommand{\len}{\textbf{l}}

\title{The Effective Ehrenpreis conjecture}
\author{Qiliang LUO}

\begin{document}
	\maketitle
	\begin{abstract}
		Let $M$ and $N$ be two closed hyperbolic Riemann surfaces.  The Ehrenpreis Conjecture (proved by Kahn-Markovic in \cite{KM-Correction})  asserts that for any  $\epsilon>0$ there are finite covers $M_\epsilon \to M$, and $N_\epsilon \to N$, such that the Teichmuller distance (in the suitable moduli space) between  $M_\epsilon$ and $N_\epsilon$ is less than $\epsilon$.  It is  natural to ask how large  the degrees of these coverings need to be to achieve that the distance between $M_\epsilon$ and $N_\epsilon$ is less than $\epsilon$. In this paper we show that there exists a constant $k>0$, depending only on $M$ and $N$, so that the covers  $M_\epsilon \to M$, and $N_\epsilon \to N$, can be chosen to have the degrees less than $\epsilon^{-k}$. We show that this bound is optimal by considering the case when $M$ and $N$ are arithmetic Riemann surfaces with the same invariant trace field which are not commensurable to each other. 
	\end{abstract}
	
	\setcounter{tocdepth}{1}
	\tableofcontents

	\section{Introduction}
	\subsection{Main Results}
	For any integer $g\geq 2$, let $\mathcal M_g$ be the moduli space of closed hyperbolic surfaces of genus $g$. For any $S_1,S_2\in\mathcal M_g$ we define the Teichmuller distance by
	\[
	d_{T}(S_1,S_2)=\min_{f:S_1\rightarrow S_2}\left\{\log K(f)\right\}
	\]
	where $f$ is quasiconformal map and $K(f)\geq 1$ the quasiconformal constant of $f$ (see \cite{A}).\par  Kahn and Markovi\'c proved in their Ehrenpreis conjecture paper \cite{KM-Correction} that for any two closed Riemann surfaces $M$ and $N$, and for any $\epsilon>0$, there exist finite coverings $M_\epsilon\rightarrow M$ and $N_\epsilon\rightarrow N$ such that $M_\epsilon$ and $N_\epsilon$ have the same genus, and satisfy 
	\[d_{T}(M_\epsilon,N_\epsilon)\leq\epsilon.\] 
	The main purpose of this paper is to provide an effective upper bound on the degree of the coverings. By improving Kahn-Markovi\'c's construction, especially in the part related to good pants homology, we prove the following theorem.
	\begin{theorem}[Effective Ehrenpreis]\label{main}
		Let $M$ and $N$ be two closed hyperbolic Riemann surfaces. There exist constants $k,\epsilon_0>0$ depending only on $M$ and $N$ such that the following holds. For any $\epsilon_0\geq\epsilon>0$, one can find finite coverings $M_\epsilon\rightarrow M$ and $N_\epsilon\rightarrow N$ of degree at most $\epsilon^{-k}$ satisfying $
			d_{T}(M_\epsilon,N_\epsilon)\leq\epsilon$.
	\end{theorem}
	
	This upper bound is optimal, as the following theorem shows. 
	\begin{theorem}\label{optimal} 
		Let $M,N$ be two arithmetic closed Riemann surfaces that have the same invariant trace field, but are not commensurable. There exist $k_0,\epsilon_0>0$ depending only on $M,N$ such that the following holds. Suppose  finite coverings $M_\epsilon\rightarrow M $ and $N_\epsilon\rightarrow N$ satisfy $
		  d_{T}(M_\epsilon,N_\epsilon)\leq\epsilon$ and $\epsilon\leq\epsilon_0$. Then the degrees of the covering maps are bigger than $\epsilon^{-k_0}$.
	\end{theorem}
	   Among other things, the proof of Theorem \ref{optimal} utilizes the theory of arithmetic Riemann surfaces, and the recent work of Doria and Paiva \cite{CN}.	According to \cite[Theorem 8.2.6, and Theorem 7.3.6]{Ari-Reid}, there exist two arithmetic Riemann surfaces $M$ and $N$ that have the same invariant trace field but different invariant quaternion algebras. Two Riemann surfaces that have different invariant quaternion algebras are not commensurable. This establishes the existence of the Riemann surfaces in Theorem \ref{optimal}.
	   \subsection{Covers with good pants decomposition}
	   The proof of Theorem \ref{main} essentially relies on the existence of covers with good pants decomposition for any closed hyperbolic Riemann surfaces.\par 
	   Let $M_0$ be a closed Riemann surface with a given pants decomposition $\mathcal C$. The reduced Fenchel-Nielsen coordinates of the pants decomposition $\mathcal{C}$ are then given by the collection of $\big(\mathbf{hl}(C),\mathbf{s}(C)\big)$ for any $C\in\mathcal C$. For any $C\in\mathcal C$, we define $\mathbf{hl}(C)$ as the half length of the geodesic curve $C$, and $\mathbf{s}(C)\in[0,\mathbf{hl}(C))$ as the reduced twisted parameter of two pants that has a boundary $C$, following \cite{KM-surface group} (see also Subsection 2.1). 
		\begin{definition}[Good Pants Decomposition]\label{good-decom}
			Let $\epsilon,R>0$. A pants decomposition $\mathcal C$ is called $(\epsilon,R)$-good if the reduced Fenchel-Nielsen coordinates satisfy the following inequalities 
			\begin{equation*}
				\begin{aligned}
					\left|\mathbf{hl}(C)-R\right|<\epsilon&\text{, and}&\left|\mathbf{s}(C)-1\right|<\frac{\epsilon}{R}
				\end{aligned}
			\end{equation*}
			for any closed curve $C\in\mathcal C$. 
		\end{definition}
		The following theorem asserts that any closed hyperbolic Riemann surface has covers with a good pants decomposition. Its proof will be presented in the remainder of this paper. 
		\begin{theorem}\label{main2}
			 Let $M$ be a closed hyperbolic Riemann surface. There exist constants $R_M,\epsilon_M,q_M>0$ such that the following holds. Suppose $R\geq R_M$, and $\epsilon_M\geq\epsilon\geq e^{-q_MR}$. Then there exists a finite covering $M_0\rightarrow M$ of degree at most $e^{50R}$ such that  $M_0$ admits a $(\epsilon,R)$-good pants decomposition. 
		\end{theorem}
		\subsection{Surfaces with good pants decomposition are close}
		The other key fact underlying the proof of Theorem \ref{main} is that, two Riemann surfaces with good pants decomposition are close in the moduli space after passing to finite covers. This fact is established in Corollary \ref{coverdeg}.\par 
		 Let $\Pi_R$ be the hyperbolic pair of pants whose three boundaries have the half length $R$. Let $S_R$ be the genus 2 hyperbolic surface that is obtained by gluing two copies of $\Pi_R$ with twist parameter $1$. By $O_R$ we denote the quotient of $S_R$ by the group of automorphisms of $S_R$. Thus any Riemann surface with $(0,R)$-good pants decomposition is a finite cover of the orbifold $O_R$.\par 
		 The following theorem is due to Kahn-Markovi\'c \cite[Theorem 2.2]{KM-surface group}. The theorem establishes that a Riemann surface with $(\epsilon,R)$-good pants decomposition is close to a finite cover of the orbifold $O_R$.
		\begin{theorem}[Kahn-Markovi\'c's Criterion]\label{criterion}
			There exist universal constants $\hat\epsilon,\hat R,c>0$ such that the following holds. Assume $M_0$ is a Riemann surface with a $(\epsilon,R)$-good pants decomposition for some $\hat\epsilon\geq\epsilon>0$, and $R>\hat R$. Then there exist a Riemann surface $M_0^\prime$ with $(0,R)$-pants decomposition, and a $(1+c\epsilon)$-quasiconformal map $f:M_0\rightarrow M_0^\prime$.
		\end{theorem}
		\begin{corollary}\label{coverdeg}
			Suppose $M_0,N_0$ are two Riemann surfaces with $(\frac{\epsilon}{3c},R)$-good pants decomposition. There exist finite covers $M_\epsilon\rightarrow M_0$ and $N_\epsilon\rightarrow N_0$ such that  $d_{T}(M_\epsilon,N_\epsilon)\leq \epsilon$, and the following holds
			\begin{equation*}
				\begin{aligned}
					deg(M_\epsilon\rightarrow M_0)\leq -3\chi(N_0)&\text{, and }&deg(N_\epsilon\rightarrow N_0)\leq -3\chi(M_0),
				\end{aligned}
			\end{equation*}
			where $\chi(M_0)$ and $\chi(N_0)$ are Euler characteristic number of the corresponding surface.
		\end{corollary}
		\begin{proof}
			First, we construct the covering maps. According to Kahn-Markovi\'c's Criterion, there exist Riemann surfaces $M_0^\prime,N_0^\prime$, and $(1+\frac{\epsilon}{3})$-quasiconformal maps 
		\begin{equation*}
			\begin{aligned}
			f:M_0\rightarrow M_0^\prime& \text{, and }&g:N_0\rightarrow N_0^\prime	
			\end{aligned}
		\end{equation*}
		where $M_0^\prime,N_0^\prime$ are covers of $O_R$. Let $S$ be the minimal common cover of $M_0^\prime$ and $N_0^\prime$ (if $O_R=\mathbb H^2/G$, $M_0^\prime=\mathbb H^2/G_1$, and $N_0^\prime=\mathbb H^2/G_2$, then $S=\mathbb H^2/(G_1\cap G_2)$, where $G_1,G_2$ are subgroups of $G$). Let 
		\begin{equation*}
			\begin{aligned}
				\tilde f:M_\epsilon\rightarrow S&\text{, and }&\tilde g:N_\epsilon\rightarrow S
			\end{aligned}
		\end{equation*}
		 be the lifting of the maps $f$ and $g$ with respect to the covering maps $S\rightarrow M_0^\prime$ and $ S\rightarrow N_0^\prime$ respectively. Thus, $M_{\epsilon}$ and $N_{\epsilon}$ are coverings of $M_0$ and $N_0$ respectively. The lifted maps $\tilde f$ and $\tilde g$ are also $(1+\frac{\epsilon}{3})$-quasiconformal. Let $h=(\tilde g)^{-1}\circ\tilde f:M_\epsilon\rightarrow N_\epsilon$. The map $h$ is a $(1+\epsilon)$-quasiconformal map. Therefore, $d_{T}(M_\epsilon,N_\epsilon)\leq \epsilon$.\par 
		  By construction $deg(M_\epsilon\rightarrow M_0)=deg(S\rightarrow M_0^\prime)$. Meanwhile,
			\begin{equation*}
				\begin{aligned}
					deg(S\rightarrow M_0^\prime)\leq deg(N_0^\prime\rightarrow O_R)=\frac{\chi(N^\prime_0)}{\chi(O_R)}=-3\chi(N^\prime_0)
				\end{aligned}
			\end{equation*}
			where the inequality comes from the group theoretic fact 
			\[[G_1:G_1\cap G_2]\leq [G:G_2].\]
			This implies the estimate about $deg(M_\epsilon\rightarrow M_0)$. By applying the same argument, we obtain the estimate about $deg(N_\epsilon\rightarrow N_0)$.
		\end{proof}
		
		\subsection{Proof of Theorem \ref{main}} 
		Let $M,N$ be two closed hyperbolic Riemann surfaces. We are going to construct the covering maps $M_\epsilon\rightarrow M$ and $N_\epsilon\rightarrow N$ such that $d_{T}(M_\epsilon,N_\epsilon)\leq\epsilon$, and estimate the degree of the covering maps.\par 
		Let $q=\min\{q_M,q_N\}$, and $\epsilon_1=3c\min\{e^{-qR_M},e^{-qR_N},\epsilon_M,\epsilon_N\}$. Therefore, provided that $\epsilon_1\geq \epsilon>0$, and $R=-q^{-1}\log(\epsilon/3)$ , the constants $\frac{\epsilon}{3c}$ and $R$ satisfy the assumptions of Theorem \ref{main2} for both $M$ and $N$. Thus, there exist finite covers $M_0\rightarrow M$ and $N_0\rightarrow N$, both of degree at most $e^{50R}$, such that $M_0$ and $N_0$ admit $(\frac{\epsilon}{3c},R)$-good pants decomposition. \par  
		By applying Corollary \ref{coverdeg} to the surfaces $M_0$ and $N_0$, we obtain finite covers $M_\epsilon\rightarrow M_0$ and $N_\epsilon\rightarrow N_0$ such that $d_T(M_\epsilon,N_\epsilon)\leq \epsilon$. Now it suffices to estimate the degree of finite covers. By Corollary \ref{coverdeg}, we have the following inequality 
		\begin{equation*}
			\begin{aligned}
				deg(M_\epsilon\rightarrow M)&=deg(M_\epsilon\rightarrow M_0)deg(M_0\rightarrow M)\\
				&\leq -3\chi(N_0)deg(M_0\rightarrow M).
			\end{aligned}
		\end{equation*}
		 Since the degrees of $M_0\rightarrow M$ and $N_0\rightarrow N$ are at most $e^{50R}=(\epsilon/3)^{-50q^{-1}}$, we obtain 
		\[
		-3\chi(N_0)deg(M_0\rightarrow M)\leq C\epsilon^{-100q^{-1}}
		\]
		for some constant $C=C(M,N)$. Let $\epsilon_0=\min\{\epsilon_1,3^{-100}C^{-q}\}$. Combining the above two estimates, we obtain 
		\[deg(M_\epsilon\rightarrow M)\leq \epsilon^{-k}\] for $k=101q^{-1}$ and $\epsilon_0\geq \epsilon$. The same estimate holds for the covering map $N_\epsilon\rightarrow N$. This completes the proof.

		\subsection{Proof of Theorem \ref{optimal}} The only fact we need about the arithmetic of $M$ and $N$ is contained in the following lemma, which establishes that the union of the length spectrum of $M$ and the length spectrum of $N$ has exponentially small gap.
		\begin{lemma}\label{ari}
			Let $M$ and $N$ be two arithmetic closed Riemann surfaces with the same invariant trace field $k$. Suppose $l$ is the length of a geodesic in $M$, and $l^\prime$ is the length of a geodesic in $N$, and $l> l^\prime$. Then 
			\[
			l/l^\prime\geq 1+ ce^{-2l}
			\]
			for some constant $c$ depending only on $M,N$.
		\end{lemma}
		\begin{proof}
		Let $M=\mathbb H^2/G$, and $N=\mathbb H^2/G^\prime$ where $G,G^\prime$ are two Fuchsian groups in $PSL_2(\mathbb R)$. Let $A\in G$, and $A^\prime\in G^\prime$ be group elements representing the length-$l$ and the length-$l^\prime$ geodesics respectively. Thus $|tr(A)|=2\cosh(l/2)$, and $|tr(A^\prime)|=2\cosh(l^\prime/2)$. Let $\xi=tr(A^2)-tr((A^\prime)^2)$. We apply the trace formula $tr(A)^2-2=tr(A^2)$ to get the equation
		\[
		4\big(\cosh(l/2)^2-\cosh(l^\prime/2)^2\big)=\xi
		\]
		By elementary estimate and Mean Value Theorem, the following holds
		\[
		e^{2l}\left(l/l^\prime-1\right)\geq 2\sinh(l)(l-l^\prime)\geq |\xi| 
		\] 
		Next, we prove that $|\xi|$ is bigger than some constant depending only on $M,N$. Let $k\subset\mathbb R$ be the invariant trace field of $M$ and $N$. Let $\phi_1=id,\phi_2,\dots,\phi_d$ be all field embedding from $k$ to $\mathbb C$. The norm of the number $\xi$ is defined by
		\[
		n(\xi)=\prod_{i=1}^d\phi_i(\xi)
		\]
		According to \cite[Theorem 1]{Ari} the following two statements hold
		\begin{enumerate}
			\item the number $\xi\neq 0$ is an algebraic integer in $k$. Therefore $n(\xi)$ is an integer and $n(\xi)\geq 1$.
			\item there exists a constant $C=C(M,N)>0$ such that for any $A\in G\cup G^\prime$, and $\phi_i\neq id$ we have $|\phi_i(tr(A^2))|\leq C$. Thus $|\phi_i(\xi)|\leq 2C$ for any $i\neq 1$.
		\end{enumerate}
		Combining the above two statements, we get $|\xi|\geq (2C)^{1-d}$. 
	Let $c=(2C)^{1-d}$. This completes the proof.
		\end{proof}
		
		Let $M,N$ be two arithmetic closed Riemann surfaces with the same invariant trace field that are not commensurable. Let $M_\epsilon\rightarrow M$, and $N_\epsilon\rightarrow N$ be finite coverings such that there exists a $e^\epsilon$-quasiconformal map $f:M_\epsilon\rightarrow N_\epsilon$. We are going to establish a lower bound on the genus of $M_\epsilon$, which implies the claimed lower bound on the degrees of the coverings $M_\epsilon\rightarrow M$, and $N_\epsilon\rightarrow N$. \par 
		 For any closed curve $\gamma$, we denote $\len(\gamma)$ to be the length of the geodesic homotopic to $\gamma$. Since $M,N$ are not commensurable, the map $f$ is not homotopic to a conformal map. Thus by \cite[Theorem 1.1]{CN} there exists a closed curve $\gamma$ in $M_\epsilon$ such that $\len(\gamma)\neq\len(f(\gamma))$, and $\len(\gamma)\leq a\log g$ for some constant $a$ depending only on $M$, and $g=g(M_\epsilon)$ the genus of the surface $M_\epsilon$. 
		
		 \par Let $l=\len(\gamma)$ and $l^\prime=\len(f(\gamma))$. Without loss of generality, we assume $l> l^\prime$. We apply Wolpert's Lemma \cite[Lemma 3.1]{W} or \cite[Corollary 5.4]{EK-dome} to get $e^{\epsilon}\geq l/l^\prime$, and apply Lemma \ref{ari} to get $l/l^\prime\geq 1+ce^{-2l}$ where $c=c(M,N)$ is the constant from the lemma. Since $l\leq a\log g$, by combining the above three estimates, we obtain
		 \begin{equation*}
		 	\begin{aligned}
		 		e^\epsilon\geq 1+cg^{-2a}&\text{, or equivalently }&g\geq \frac{c^{1/2a}}{(e^\epsilon-1)^{1/2a}} 
		 	\end{aligned}
		 \end{equation*}
		Let $k_0=1/4a$, and $\epsilon_0=\min\{1,c^2(e-1)^2\}$. Thus, we obtain $g\geq \epsilon^{-k_0}$ when $\epsilon_0\geq \epsilon$. This completes the proof.
		\subsection{Outline of the paper}
		It remains to prove Theorem \ref{main2}. That is, for any suitable $\epsilon,R$ we need to construct a finite covering $M_{\epsilon,R}\rightarrow M$ that admits a $(\epsilon,R)$-good pants decomposition. The covering map is constructed by gluing a family of immersed pants in $M$. Let $\Gamma_{\epsilon,R}$ be the set of  geodesic curves that have half length $\epsilon$-close to $R$. Let $\PPi_{\epsilon,R}$ be the set of orientation-preserving immersed pair of pants that have half cuff lengths $\epsilon$-close to $R$.

		\begin{enumerate}
			\item In Section 2, we construct the needed covering map by assuming the Good Pants Homology Theorem.  Let $\mu_1$ be the family of pants  that is the formal sum of all good pants in $\PPi_{\epsilon,R}$ with unit coefficients. We will apply the Good Pants Homology Theorem to correct the multi-pants $\mu_1$ such that the corrected multi-pants $\mu$ can be glued to form the needed covering map. The degree of the constructed covering map is then bounded by the number of pants in $\mu$. This completes the proof of Theorem \ref{main2}.
		\end{enumerate}
		The Good Pants Homology Theorem is the central technical contribution of the paper, which provides an effective version of \cite[Theorem 3.4]{KM-Correction} (the precise meaning of this effectiveness is discussed in Subsection 2.4). Let 
		\[
		\partial:\PPi_{\epsilon,R}\rightarrow\mathbb Z\Gamma_{\epsilon,R}
		\]
		be the boundary map. Briefly stated, the theorem wants to find a semirandom map 
		\[
		\Phi:\Gamma_{\epsilon,R}\rightarrow\mathbb Q\PPi_{\epsilon,R}
		\]
		such that $\partial\Phi(\nu)=\nu$ for any homologically trivial good multi-curve $\nu\in\mathbb Z\Gamma_{\epsilon,R}$. We will prove the Good Pants Homology Theorem by constructing the needed map $\Phi$ explicitly. The proof occupies the text of Sections 3-7.
		
		\begin{enumerate}\setcounter{enumi}{1} 
			\item In Sections 3 and 4, we review hyperbolic geometry and the theory of semirandom maps, respectively.  The basic tools in Section 3 will support the geometric side of the construction of the map $\Phi$. To verify that the constructed map $\Phi$ is semirandom, we will estimate the semirandom norm at each step of the construction.
			\item Sections 5 and 6 are devoted to establishing the connection between the group homology and the good pants homology. Let $G$ be the fundamental group of the surface $M$. We construct maps $\mathcal R_G$ and $\mathcal R_{G\times G}$ defined on subsets of $G$ and $G\times G$, respectively, such that the following diagram commutes.
				\[\begin{tikzcd}
		\mathbb Z(G\times G)\arrow[d,"\partial"] \arrow[r,"\mathcal{R}_{G\times G}"]& \mathbb Q\PPi_{\epsilon,R}\arrow[d,"\partial"]\\
		\mathbb ZG\arrow[r,"\mathcal R_G"]&\mathbb Q\Gamma_{\epsilon,R}
	\end{tikzcd}\]
			 This is established in the Group-to-Pants Theorem.
			\item In Section 7, we prove the Good Pants Homology Theorem by first encoding good curves in group homology, then applying the Group-to-Pants Theorem.
		\end{enumerate}

	\section{Constructing covers of a Riemann surface}
		In this section, we prove Theorem \ref{main2} by assuming the Good Pants Homology Theorem (Theorem \ref{good pants}), which is proved in Section 7. Let $M$ be a closed hyperbolic Riemann surface. For given suitable $\epsilon$ and $R$, we construct a covering $M_{\epsilon,R}\rightarrow M$ that admits a $(\epsilon,R)$-good pants decomposition. This covering is constructed by gluing good pants from $M$. Consequently, the degree of the covering is bounded above by the number of such pants employed.\par 
		This section proceeds as follows. 
		\begin{enumerate}
			\item In Subsection 2.1, we establish the following criterion: an evenly distributed and equidistributed family of pants can be glued to form a covering that admits a good pants decomposition.
			\item  Subsections 2.2 and 2.3 then show that the family of pants taking every good pants exactly once is equidistributed and almost evenly distributed. This is established in Theorem \ref{equidistribution}, which is a more quantitive version of \cite[Theorem 3.1]{KM-Correction}. 
			\item In Subsection 2.4, we introduce the Good Pants Homology Theorem and correct (by applying the theorem) the above family of pants to the one that is evenly distributed. In particular, we explain the difference between our Good Pants Homology Theorem and Kahn-Markovic's.
			\item  In Subsection 2.5, we prove that the corrected multi-pants meets the criterion, thereby completing the construction and the proof of Theorem \ref{main2}.
		\end{enumerate}

	 \subsection{Pants in a Riemann surface}
	 In this subsection, we introduce basic definitions and notations about immersed pair of pants in a Riemann surface $M$. Then prove the Covering Constructing Lemma, which is the criterion mentioned at the beginning of the section.
	 \par Let $\gamma$ be a geodesic arc or a closed geodesic arc. By $\len(\gamma)$ we denote the length of $\gamma$. For any geodesic arc $\alpha$, by $i(\alpha)$ we denote the unit initial vector of $\alpha$ and by $t(\alpha)$ we denote the unit terminal vector.
	 \par 
	 Let $\Pi_0$ be a fixed topological pair of pants. By $C_1,C_2,C_3$ we denote three boundaries of $\Pi_0$, called cuffs. By $S_i$ we denote the simple arcs $S_i:[0,1]\rightarrow\Pi_0$ connecting $C_i,C_{i+1}$ with $S_i(0)\in C_i$.  An immersed pair of pants in $M$ is an immersing map 
	 	\[
		\Pi:\Pi_0\rightarrow M\]
		 up to homotopy and the action of $MCG(\Pi_0)$ such that $\Pi(C_i)$ is the geodesic curve and $\Pi(S_i)$ is the minimal length geodesic arc for any $i=1,2,3$. 
		 \par For any oriented closed curve $\gamma$, by $\sqrt{\gamma}$ we denote the set of all pairs of antipodal points in $\gamma$ and denote $\mathbf{hl}(\gamma)=\len(\gamma)/2$. By $N^1(\gamma)$ we denote the unit normal bundle of $\gamma$. By  $N^1(\sqrt{\gamma})$ we denote the set of all pairs of unit normal vectors of $\gamma$ supported at two antipodal points and on the same side of $\gamma$. Thus $N^1(\sqrt{\gamma})$ has two connected components, the right side $N^1_+(\sqrt{\gamma})$ and the left side $N^1_-(\sqrt{\gamma})$. Suppose that $\gamma=\Pi(C_1)$ is a boundary of the immersed pair of pants. We define the foot of $\Pi$ on $\gamma$ to be
		 \[
		 Ft_\gamma(\Pi)=\{i(S_1),-t( S_2)\}\in N^1_+(\sqrt{\gamma})
		 \]
		 \begin{definition}\label{twist}
		 	We identify $\sqrt{\gamma}$ with the Euclidean tori $\mathbb R/\mathbf{hl}(\gamma)\mathbb Z$ as a measured metric space. Suppose that $\Pi_1$ is an immersed pair of pants that has a boundary $\gamma$ and $\Pi_2$ is an immersed pair of pants that has a boundary $\overline\gamma$. Thus $Ft_\gamma(\Pi_1)\in N^1_+(\sqrt{\gamma})$ and $Ft_{\overline\gamma}(\Pi_2)\in N^1_-(\sqrt{\gamma})$. Let $a$ be the support point of $Ft(\Pi_1)$ at $\sqrt{\gamma}$ and $b$ be the support point of $Ft(\Pi_2)$. Then their twist parameter $\mathbf s$ is defined by
		 	\[
		 	\mathbf s_\gamma(\Pi_1,\Pi_2)=\len(\gamma_{ab})
		 	\] where $\gamma_{ab}$ is the subarc of $\sqrt{\gamma}$ from $a$ to $b$.
		 \end{definition} 
		 By $\Gamma$ we denote the space of all oriented closed geodesics in $M$. By $\PPi$ we denote the space of all immersed pair of pants in $M$.  By  $N^1(\sqrt{\Gamma})$ we denote the disjoint union of all $N^1(\sqrt{\gamma})$. Let $\mathbb Z\Gamma$ be the set of all formal sums of curves, identifying $\overline\gamma$ with $-\gamma$. We have the boundary map
		\begin{equation*}
			\begin{aligned}
				\partial :\PPi\rightarrow\mathbb Z\Gamma&\text{, defined by}&\Pi\mapsto\sum_{i=1}^3\Pi(C_i).
			\end{aligned}
		\end{equation*}
		and the following finer boundary map
		\begin{equation*}
			\begin{aligned}
				\hat\partial :\PPi\rightarrow\mathbb ZN^1\left(\sqrt{\Gamma}\right)&\text{, defined by}&\Pi\mapsto\sum_{i=1}^3\bigg(\Pi(C_i),Ft_i(\Pi)\bigg)
			\end{aligned}
		\end{equation*}
		A family of immersed pair of pants can be considered as a multi-pants $\mu\in\mathbb Z^+\PPi$. For any multi-pants $\mu\in \mathbb Z\PPi$ represented by $
			\mu=\sum\Pi_n$, by $\partial \mu$ we denote the linearly extension of the boundary map. Namely, we define $\partial \mu=\sum\partial \Pi_n\in \mathbb Z\Gamma$. For any geodesic $\gamma$ and multi-pants $\mu$, by $\hat\partial _\gamma(\mu)$ we denote the counting measure on $N^1_+\left(\sqrt{\gamma}\right)$ with respect to the finite subset $\hat\partial (\mu)|_{N^1_+\left(\sqrt{\gamma}\right)}$(counting multiplicity). This measure characterizes the distribution of pants in the multi-pants $\mu$ that has a boundary $\gamma$. \par 
			Let $\Gamma_{\epsilon,R}$ be the subset of $\Gamma$ consisting of curves satisfying that
		 \[
		 \left|\mathbf{hl}(\gamma)-R\right|\leq\epsilon
		 \]By $\PPi_{\epsilon,R}$ we denote the space of all immersed pair of pants such that the half length of their boundaries is  $\epsilon$ close to $R$.\par 
		 Before stating the Covering Constructing Lemma, we recall the definition of equivalent measures. Let $X$ be a metric space. Let $\mu,\nu$ be two measures such that $\mu(X)=\nu(X)$ and let $\delta>0$ be a constant. Suppose that for every Borel set $A\subset X$ we have $\mu(A)\leq\nu(N_\delta(A))$. Then we say that $\mu$ and $\nu$ are $\delta$-equivalent measures.
		\begin{lemma}[Covering Constructing]\label{glue-crit}
			 Suppose that a multi-pants $\mu\in\mathbb Z^+\PPi_{\epsilon,R}$ satisfies following conditions
			\begin{enumerate}
				\item (Evenly distributed) $\partial \mu=0\in\mathbb Z\Gamma_{\epsilon,R}$.
				\item (Equidistributed) for any $\gamma\in\Gamma_{\epsilon,R}$, the measure $\hat\partial _\gamma(\mu)$ on $N^1_+\left(\sqrt{\gamma}\right)$ is $\frac{\epsilon}{2R}$-equivalent to an Euclidean measure.
			\end{enumerate}
			Then one can glue these multi-pants to get a cover $M_{\epsilon,R}\rightarrow M$ that admits a $(\epsilon,R)$-good pants decomposition.
		\end{lemma}
		\begin{proof}
			Let $\gamma$ be an oriented closed geodesic. Let $\iota$ be the distance of one translation on the Euclidean tori $\sqrt{\gamma}$. Since the Euclidean measure is invariant under the action of $\iota$, the equidistributed property implies that $\iota_*\hat\partial _\gamma\mu$ is $\frac{\epsilon}{2R}$-equivalent to $K_\gamma d\lambda_\gamma$ for some constant $K_\gamma$ where $d\lambda_\gamma$ is the Lebesgue measure on $\sqrt{\gamma}$. On the other hand, the measure $\hat\partial _{\overline\gamma}\mu$ is equivalent to the Euclidean measure $K_{\overline\gamma}d\lambda_\gamma$ for some constant $K_{\overline\gamma}$. The evenly distributed property implies that $K_\gamma=K_{\overline\gamma}$. Therefore, the measure $\iota_*\hat\partial _\gamma\mu$ is $\frac{\epsilon}{R}$-equivalent to the measure $\hat\partial _{\overline\gamma}\mu$. By Hall's Marriage Theorem, see \cite[Theorem 3.2]{KM-surface group}, there exists a pairing $
			\sigma_\gamma$ between the foots on $\hat\partial _\gamma\mu$ and the foots on $\hat\partial _{\overline\gamma}\mu$ such that, for any $\Pi\in\mu$ that has a boundary $\gamma$, we have
			\begin{equation}\label{E00}
				\left|\mathbf s_\gamma\left(\Pi,\sigma_\gamma\Pi\right)-1\right|\leq\frac{\epsilon}{R}
			\end{equation}
			\par We construct the needed cover by gluing a pair of pants $\Pi$ in the multi-pants $\mu$ with the pair of pants $\sigma_\gamma\Pi$ along the boundary $\gamma$. This obviously gives a covering mapping. Since each gluing block satisfies the half length condition and the gluing $\sigma$ satisfies the inequality (\ref{E00}), the cover satisfies the conditions from the good pants decomposition.
		\end{proof}

	\subsection{Generating immersed pants}
	In this subsection, we study how to construct immersed pair of pants. The construction is called the $\theta$-graph construction. As a rapid application, we encode each pair of pants by the third connection whose distribution is well understood.
	\par  Let $p,q$ be two points in $M$, and let $\alpha_1,\alpha_2,\alpha_3$ be three geodesic arcs starting at $p$ and ending at $q$. The union of three geodesic arcs is called a $\theta$-graph if the triples of unit vectors $(i(\alpha_1),i(\alpha_2),i(\alpha_3))$ and $(t(\alpha_1),t(\alpha_2),t(\alpha_3))$ have opposite cyclic ordering with respect to the orientation on $T^1M_p$ and $T^1M_q$. A $\theta$-graph in $M$ is homotopic to a unique immersed pair of pants in $M$. The following proposition is a special case of the above $\theta$-graph construction. 

	\begin{proposition}[Third Connection Construction]
		Let $\gamma$ be a closed geodesic curve and $
	\eta$ be an orthogeodesic arc connecting the right side of $\gamma$ and right side of $\gamma$. Then there exists a unique immersed pair of pants that is homotopic to the immersed graph $\gamma\cup\eta$ in $M$, denoted by $Thd(\gamma,\eta)$.
	
	\end{proposition}
	
	By $\PPi_\gamma$ we denote the space of all immersed pants that have $\gamma$ as a boundary. By $Conn_\gamma$ we denote the space of all orthogeodesics connecting the right side of $\gamma$ and the right side of $\gamma$. The third connection construction defines a map
	\[
	Thd_\gamma:Conn_\gamma\rightarrow\PPi_\gamma.
	\]
	 Conversely, for any oriented closed geodesic $\gamma$ and any pair of pants $\Pi\in\PPi_\gamma$, the third connection $\eta$ is the unique separating arc in $\Pi$ connecting $\gamma$ and $\gamma$ itself with minimal length. Therefore, the third connection map is a bijective map for any $\gamma$. 
	 \par Next, we investigate the distribution of the third connections. Let $A,B$ be two oriented geodesic segments on $M$, and let $0<L_0<L_1$. We define
	\[
	Conn_{A,B}(L_0,L_1)
	\]
	to be the space of all orthogeodesic connecting the right side of $A$ and the right side of $B$ such that
	 $len(\eta)\in [L_0,L_1]$. The following Connecting Principle was proved in \cite[Theorem 11.3]{KM-Correction} by applying the exponential mixing property of the geodesic flow on $M$. 
		
	 \begin{connecting}
		There exist constants $q_0=q_0(M)$ satisfying the following statement. For any $L>0$, let $\delta=\exp(-q_0L)$. Suppose that $len(A)=len(B)=\delta^2$, then
		\begin{equation}\label{Connecting Principle}
			\#Conn_{A,B}(L,L+\delta^2)=E_M\exp(L)\delta^6\big(1+O(\delta)\big).
		\end{equation}
		where $E_M$ and $O$ depending only on $M$. Moreover, the constant $q_0$ satisfies that $q_0=q/40$ where $q$ is the exponential mixing constant of the geodesic flow on $M$, see \cite[Theorem 11.3]{KM-Correction}. Furthermore, the mixing constant $q$ satisfies that $q= C^\prime\lambda_1(M)$ for some universal constant $C^\prime$, see \cite[Proposition 5.1]{Mo}. Combining the above two results, we obtain
		\[
		q_0=C\lambda_1(M)
		\]
		for some universal constant $C$.
	\end{connecting}
	 The following corollary follows directly from Connecting Principle.
	 
	\begin{corollary}
		\label{testimate}
	Let $q_0,E_M$ be the constants from Connecting Principle. For any $W>0$, let $\delta=e^{-q_0W}$. Thus, for any region
	 \[\Omega\subset\gamma\times\gamma\times [W,\infty),\]
	 we have
	 \begin{equation*}
	 	\begin{aligned}
	 		\#\Big(Conn_{\gamma}\cap \Omega\Big)=&\left(\int_\Omega E_M\exp(w)dwd\lambda_\gamma^2\pm\int_{N_{3\delta^2}(\partial \Omega)}E_M\exp(w)dwd\lambda_\gamma^2\right)\\
	 		&\cdot\big(1+O(\delta)\big)
	 	\end{aligned}
	 \end{equation*}
	 where  $O$ only depends on $M$.
	\end{corollary}

	\subsection{The Distribution of good pants}
	In this subsection, we prove that the multi-pants $\mu_1\in \mathbb Z\PPi_{\epsilon,R}$ that assigns to each pants the value 1 is equidistributed and almost evenly distributed. 
	
	\begin{theorem}\label{equidistribution}
		Let $q_0$ be the constant from Connecting Principle. For any sufficiently large $R$ and $1>\epsilon\geq e^{-2q_0R} $, the multi-pants $\mu_1\in\mathbb Z\PPi_{\epsilon,R}$ that assigns to each pants the value 1 satisfies that 
		\begin{enumerate}
			\item for every $\gamma\in \Gamma_{\epsilon,R}$, let $
		K_\gamma$ be the total number of all pants in $\PPi_{\epsilon,R}$ that have $\gamma$ as a boundary. Then $K_\gamma\asymp \epsilon^2Re^R$.
			\item (almost evenly distributed) for any $\gamma\in \Gamma_{\epsilon,R}$, we have 
		\[
			\left|\frac{K_\gamma}{K_{\overline\gamma}}-1\right|\leq C\delta
			\]
			where $\delta=e^{-q_0R}$ and $C$ is a universal constant.
		\item (equidistributed) for any $\gamma\in\Gamma_{\epsilon,R}$, and any subarc $I\subset\sqrt{\gamma}$ satisfying $\len(I)\geq \delta^2$, we have 
		\[
		\hat\partial _\gamma\mu_1(I)=\frac{K_\gamma}{\mathbf{hl}(\gamma)}\len(I)\cdot\left(1+O(\delta)\right).
		\]
		\end{enumerate}
	\end{theorem}
	
	\begin{proof}
	We investigate the measure $\hat\partial _\gamma\mu_1$ by describing the region where the third connections lie. Suppose that $\gamma\in \Gamma_{\epsilon,R}$ is an oriented geodesic curve and $I$ is a subarc of $\sqrt{\gamma}$. We define the space $\mathcal R(\gamma,I)\subset\gamma\times\gamma\times\mathbb R^+$ as the set of $(x,y,w)$ for which we have 
	\begin{enumerate}
		\item Let $\sigma_1,\sigma_2$ be two segments of $\gamma\backslash\{x,y\}$. Then for any $i=1,2$
		\[
		\bigg|h\left(\len(\sigma_i),w\right)-2R\bigg|\leq 2\epsilon
		\]
		where $h$ is a function determined by the hyperbolic trigonometry such that $h\left(\len(\sigma_i),w\right)$ equal to the length of the closed geodesic $[\eta\sigma_i]$. Namely $
		\cosh\big(h(a,b)/2\big)=\sinh(a/2)\sinh(b/2)$.
		\item The pair of antipodal points consisting of midpoints of $\sigma_1,\sigma_2$ lies in the interval $I$. 
		
	\end{enumerate}
	Thus a pair of pants $\Pi\in\PPi_\gamma$ lies in $\PPi_{\epsilon,R}$ and $Ft_\gamma(\Pi)\in I$ if and only if its third connection $\eta\in  \mathcal R(\gamma,I)$. This implies the following equation
		\begin{equation}
			\hat\partial _\gamma\mu(I)=\#\bigg(Conn_{\gamma}\cap\mathcal R(\gamma,I)\bigg).
		\end{equation}
		The volume of the regions with respect to the measure $d\lambda^2dw$ on $\gamma\times\gamma\times\mathbb R^+$ satisfies that
		\begin{equation*}
			\begin{aligned}
				Vol\big(\mathcal R(\gamma,I)\big)\asymp \epsilon^2\len(I)&\text{, and }&Vol\left(N_{3\delta^2}\left(\partial \mathcal R(\gamma,I)\right)\right)\asymp \epsilon^2\delta^2+\epsilon\len(I)\delta^2,
			\end{aligned}
		\end{equation*}
		where $\delta=e^{-q_0R}$. If $\len(I),\epsilon>\delta^2$, then we have
		\begin{equation*}
			Vol\left(N_{3\delta^2}\left(\partial \mathcal R(\gamma,I)\right)\right)=Vol\big(\mathcal R(\gamma,I)\big)O(\delta).
		\end{equation*}
		Combining the above estimates with Corollary \ref{testimate}, we obtain that	
	 \begin{equation}\label{E03}
	 	\begin{aligned}
	 		\hat\partial _\gamma\mu_1(I)=&\int_{\mathcal R(\gamma,I)} E_M\exp(w)dwd\lambda_\gamma^2\cdot \big(1+O(\delta)\big).
	 	\end{aligned}
	 \end{equation}
	 In particular, since $w\approx R+2\log 2$, by setting $I=\sqrt{\gamma}$, we get
	 \[
	 K_\gamma=\hat\partial _\gamma\mu_1(\sqrt{\gamma})\asymp\epsilon^2Re^R.
	 \]
	 This proves the first property.  Moreover, since the midpoints map
	 \begin{equation*}
	 	\begin{aligned}
	 		\mathcal R(\gamma,\gamma)\rightarrow \sqrt{\gamma}&&(x,y,w)\mapsto\text{the midpoints of } \sigma_1,\sigma_2
	 	\end{aligned}
	 \end{equation*}
	 is translation invariant, by applying the Fubini Theorem, we have
	 \[
	 \int_{\mathcal R(\gamma,I)} E_M\exp(w)dwd\lambda_\gamma^2=k\text{ }\len(I)
	 \]
	 for some constant $k=k\big(\len(\gamma),\epsilon,R\big)$. Combining this with the estimate (\ref{E03}), we obtain that\[\hat\partial _\gamma\mu_1(I)=k\text{ }\len(I)\cdot \big(1+O(\delta)\big).\]
	 This proves the third property, because \begin{equation}\label{E02}
	 	K_\gamma=k\text{ }\mathbf{hl}(\gamma)\cdot \big(1+O(\delta)\big).
	 \end{equation}
	 On the other hand, we claim that the estimate (\ref{E02}) also implies the almost evenly distributed of the multi-pants $\mu_1$. 
	In fact, since $k$ depends only on the length of the geodesic curve $\gamma$, we have $k(\gamma)=k(\overline\gamma)$ and $K_\gamma=K_{\overline\gamma}\cdot\left(1+O(\delta)\right)$. This completes the proof. 
	\end{proof}
	The following is a corollary of the almost evenly distributed property of the multi-pants $\mu_1$.
	\begin{corollary}
		There exists a universal constant $C$ such that for any sufficiently large $R>0$ and $1\geq\epsilon\geq e^{-2q_0R}$ the following holds. There exists a multi-curve $\nu\in\mathbb Z\Gamma_{\epsilon,R}$ such that $\nu=\partial \mu_1$ and
	\begin{equation}\label{E06}
		|\nu(\gamma)|\leq CK_\gamma\delta
	\end{equation}
	for any $\gamma\in\Gamma_{\epsilon,R}$, where $\delta=e^{-q_0R}$.
	\end{corollary}

	\subsection{The Good Pants Homology Theorem}
	In this subsection, we state the Good Pants Homology Theorem, which is proved in Section 7. This theorem is the central technical contribution of the paper. It provides an effective version of \cite[Theorem 3.4]{KM-Correction}.  
  
\begin{theorem}[Good Pants Homology]\label{good pants}
	There exist constants $R_1,\epsilon_1,q_1>0$ depending only on $M$ such that for any $R\geq R_1$ and $\epsilon_1\geq\epsilon\geq e^{-q_1R}$ the following holds. There exists a natural number $N_R\leq e^{30R}$ (depending only on $R$) and a map
	\[
	\Phi:\Gamma_{\epsilon,R}\rightarrow\frac{\mathbb Z\PPi_{\epsilon,R}}{N_R},
	\]
	such that the map $\Phi$ satisfies the following properties:
\begin{enumerate}
	\item \text{Homology condition:}  There exists a $\mathbb Z$-linear map 
	\[
	\Psi:H_1(M,\mathbb Z)\rightarrow \frac{\mathbb Z\Gamma_{\epsilon,R}}{N_R}
	\]
	such that for any good curve $\gamma\in\Gamma_{\epsilon,R}$, we have 
	\[
	\partial \Phi(\gamma)=\gamma-\Psi(H(\gamma)),
	\]
	where $H(\gamma)\in H_1(M,\mathbb Z)$ is the homology class of $\gamma$.
	\item \text{Randomness condition:} For any good pants $\Pi\in\PPi_{\epsilon,R}$, the following holds
	\[
		\sum_{\gamma\in\Gamma_{\epsilon,R}}|\Phi(\gamma)|(\Pi)\leq e^{-R}(R\epsilon^{-1})^m.
	\]
\end{enumerate}
\end{theorem} 
The Good Pants Homology Theorem is used to correct the almost evenly distributed multi-pants $\mu_1$ to an evenly distributed one.

\begin{definition}[Corrected multi-pants]
	Suppose $\epsilon,R$ satisfy the conditions from Theorem \ref{good pants}. We define 
\begin{equation*}
	\begin{aligned}
		\mu_2=\mu_1-\Phi(\nu)\in\frac{\mathbb Z\PPi_{\epsilon,R}}{N_R} &&\text{, and }\mu=N_R\cdot \mu_2\in\mathbb Z\PPi_{\epsilon,R}.
	\end{aligned}
\end{equation*}
\end{definition}
\begin{proposition}
	The multi-pants $\mu$ is evenly distributed.
\end{proposition}
\begin{proof}
	Since $\nu=\partial \mu_1$, the multi-curve $\nu$ has trivial homology. That is $H(\nu)=0$.  By applying the homology condition in Theorem \ref{good pants}, the following holds
	\begin{equation*}
		\begin{aligned}
			\partial  \mu=&N_R\cdot (\partial \mu_1-\partial \Phi(\nu))\\
			=&N_R\cdot(\nu-\partial \Phi(\nu))=0.
		\end{aligned}
	\end{equation*}
\end{proof}
It now suffices to prove that the corrected multi-pants $\mu$ is equidistributed and to estimate the number of pants in $\mu$. Before diving into the proof of Theorem \ref{main2}, we briefly explain the difference between our Good Pants Homology Theorem and Kahn-Markovic's \cite[Theorem 3.4]{KM-Correction}. The effectiveness of our theorem  has a threefold meaning:
\begin{enumerate}
	\item Firstly, it is necessary to specify the admissible range for the constants $\epsilon,R$ under which this theorem holds. We prove that this theorem holds when $R\geq R_1$ and $\epsilon_1\geq\epsilon\geq e^{-q_1R}$ for some constants $R_1,\epsilon_1,q_1>0$ depending only on $M$.
	\item Secondly, to control the number of pants in the corrected family $\mu$, we must establish a version of good pants homology with effectively controlled coefficients, rather than $\mathbb Q$-coefficients as was done in \cite{KM-Correction}.
	\item Finally, we must provide a complete description of the asymptotic behavior of the randomness condition of the correction map $\Phi$ as $\epsilon\rightarrow 0$, and $R\rightarrow\infty$, rather than only as $R\rightarrow\infty$ as was done in \cite{KM-Correction}.
\end{enumerate}

\begin{remark}\label{namegp}
	The name "Good Pants Homology Theorem" is justified by the homology condition for the following reason. The good pants homology group $\Omega_1$ is defined as formal sums of good curves modulo the boundaries of formal sums of good pants, namely,
 \[
 \Omega_1=\mathbb Z\Gamma_{\epsilon,R}/\partial(\mathbb Z\PPi_{\epsilon,R}).
 \] Thus, there is a natural morphism from the good pants homology group $\Omega_1$ to the standard homology group. In Theorem \ref{good pants}, we establish an inverse $\mathcal R$ to the $\frac{\mathbb Z}{N_R}$-coefficient natural morphism. Therefore, the $\frac{\mathbb Z}{N_R}$-coefficient $(\epsilon,R)$-good pants homology is isomorphic to the $\frac{\mathbb Z}{N_R}$-coefficient standard homology, a highly nontrivial result.  
\end{remark}
	\subsection{Proof of Theorem \ref{main2}}
	In this subsection, we prove that there exist constants $R_M,\epsilon_M,q_M$ such that for any $R\geq R_M$ and $\epsilon_M\geq \epsilon\geq e^{-q_MR}$, the corrected multi-pants $\mu$ is equidistributed and the total number of pants in $\mu$ is less than $e^{50R}$. Thus, by the Covering Constructing Lemma, we can glue the multi-pants $\mu$ to get a finite cover satisfying the conditions in Theorem \ref{main2}. \par 
	  The following lemma establishes that the coefficients of the correction term $\Phi(v)$ are very small. 
	\begin{lemma}
		There exist constants $R_2,q_2>0$ depending only on $M$ such that for any $R\geq R_2$, and  $\epsilon\geq e^{-q_2R}$, the following estimate holds
	 \begin{equation}\label{E08}
	 	\Phi(\nu)(\Pi)\leq \sqrt{\delta}<1,
	 \end{equation}
	 where $\delta=e^{-q_0R}$.
	\end{lemma}
	\begin{proof}
		By applying Estimate (\ref{E06}), the following holds for any $\Pi\in\PPi_{\epsilon,R}$
	 \begin{equation*}
	 \begin{aligned}
	 	 	\Phi\nu(\Pi)&=\sum_{\gamma\in\Gamma_{\epsilon,R}}\nu(\gamma)\cdot(\Phi(\gamma))(\Pi)\\
	 	 	&\leq \sum_{\gamma\in\Gamma_{\epsilon,R}} C_1K_\gamma\delta\cdot\left|\Phi(\gamma)\right|(\Pi).
	 \end{aligned}
	 \end{equation*}
	By Theorem \ref{equidistribution} we have $K\asymp K_\gamma$, where $K=\epsilon^2Re^R$. Thus there exists a constant $C_2=C_2(M)$ such                          that $K_\gamma\leq C_2 K$. Combining this with the above estimate, we obtain 
	\[
	\Phi\nu(\Pi)\leq C_1C_2K\delta\sum_{\gamma\in\Gamma_{\epsilon,R}} \left|\Phi(\gamma)\right|(\Pi).
	\]
	By applying the randomness condition in Theorem \ref{good pants}, the following holds
	\[
	\Phi\nu(\Pi)\leq (C_1C_2K\delta)\left( e^{-R}(R\epsilon^{-1})^{m}\right)=\left(C_1C_2R^{m+1}e^{-q_0R/2}\epsilon^{-m+2}\right)\sqrt{\delta}.
	\]
	Let $q_2=q_0/(4m+8)$, and $R_2$ be sufficiently large such that
	\[
	C_1C_2R^{m+1}e^{-q_0R/2}\epsilon^{-m+2}\leq1.
	\]
	This completes the proof.
	\end{proof}
	 We are going to check that the multi-pants $\mu$ is equidistributed, and estimate the total number of pants in $\mu$.  Since there are $K_\gamma$ pants that have $\gamma$ as a boundary, by applying Estimate (\ref{E08}), we get 
	 \[
	 (\partial \Phi(\nu))(\gamma)\leq K_\gamma\sqrt{\delta}.\] Combining this with the third property from Theorem \ref{equidistribution}, for any subarc $I\subset\sqrt{\gamma}$ satisfying $\len(I)\geq \delta^{1/4}$, the following estimates hold
	 \begin{equation*}
	 	\begin{aligned}
	 		\hat\partial _\gamma\mu_2(I)&=\hat\partial _\gamma\mu_1(I)-\hat\partial _\gamma\Phi\nu(I)\\
	 		&=\frac{K_\gamma}{\mathbf{hl}(\gamma)}\len(I)\cdot\left(1+O(\delta)\right)+O\left(K_\gamma\sqrt{\delta}\right)\\
	 		&=\frac{K_\gamma}{\mathbf{hl}(\gamma)}\len(I)\cdot\left(1+O(\delta^{1/4})\right).
	 	\end{aligned}
	 \end{equation*}
	This implies that $\mu_2$ is equidistributed. Since $\mu=N_R\mu_2$, $\mu$ is also equidistributed. On the other hand, we have the following estimate
	 \begin{equation*}
	 	\begin{aligned}
	 		\mu(\PPi)= N_R\mu_2(\PPi)\leq 2N_R\mu_1(\PPi),
	 	\end{aligned}
	 \end{equation*}
	where the inequality comes from Estimate (\ref{E08}). Since 
	\[\mu_1(\PPi)\leq\#\PPi_{1,R}\asymp e^{3R},\] the total number of pants in $\mu$ is less than $e^{50R}$. 

\section{Preliminaries on hyperbolic geometry}
The proof of the Good Pants Homology Theorem is constructive.  In this section, we review some basic tools on hyperbolic geometry, which support the construction. 
\subsection{The connection set and the Connection Lemma}
 
For any two unit vectors $u_1,u_2\in T^1M$ supported at the same point, we denote $\Theta(u_1,u_2)$ to be the non-oriented angle between $u_1$ and $u_2$.
\begin{definition}(Connection set)
	 Let $L,\epsilon>0$ be two constants and $u,v$ be any two unit vectors. By $Conn_{\epsilon,L}(u,v)$ we denote the set of unit speed geodesic segments $\gamma:[0,l]\rightarrow M$ such that
\begin{enumerate}
	\item the unit vector $u$ is supported at $\gamma(0)$, and the vector $v$ is supported at $\gamma(l)$. And $\Theta(u,i(\gamma)),\Theta(v,t(\gamma))\leq\epsilon$.
	\item $|l-L|\leq\epsilon$
\end{enumerate}
By $Conn_{\epsilon,<L}(u,v)$ we denote the set of geodesic arcs satisfying the above first condition and $l< L$.
\end{definition}
 \par 
 The following lemma is a weaker version of Connecting Principle.
\begin{lemma}[Connection Lemma]
	Let $q_0$ be the constant from Connecting Principle. There exists a constant $L_0=L_0(M)>0$ such that the following holds. Let 
	\[L(\epsilon)=\max\{-q_0^{-1}\log\epsilon,-\log\epsilon,L_0\}.\]
	Suppose $L\geq L(\epsilon)$. Then the connection ser $Conn_{\epsilon,L}(u,v)$ is nonempty for any unit vectors $u$ and $v$. Moreover, there exists $C>0$ depending only on $M$ such that
	\[
	\#Conn_{\epsilon,L}(u,v)\geq Ce^{L}\epsilon^{3}.
	\]
\end{lemma}
\begin{proof}
	Let $A_u:[-\epsilon/3,\epsilon/3]\rightarrow M$ be the geodesic segment with $A_u^\prime(0)=\sqrt{-1}u$. Let $B_v:[-\epsilon/3,\epsilon/3]\rightarrow M$ be the geodesic segment with $B_v^\prime(0)=-\sqrt{-1}v$. For any orthogeodesic 
	\[X\in Conn_{A,B}(L-\epsilon/3,L+\epsilon/3),\] by moving the initial and terminal points of $X$ to the base point of $u$ and $v$, respectively, we obtain a geodesic arc $X^\prime\in Conn_{\epsilon,L}(u,v)$. Thus,
	\[
	\#Conn_{\epsilon,L}(u,v) \geq\#Conn_{A,B}(L-\epsilon/3,L+\epsilon/3).
	\]
	By Connecting Principle, there exists a constant $L_0$ such that the following estimate holds
	\[
	\#Conn_{A,B}(L-\epsilon/3,L+\epsilon/3)\geq Ce^L\epsilon^3>0
	\] 
	 provided $L\geq L_0$ and $\epsilon\geq e^{-q_0L}$. This completes the proof.
\end{proof}

\subsection{The Chain Lemma and the Right Angle Chain Lemma}
Let $\alpha_1,\dots,\alpha_n$ be a sequence of consecutive geodesic arcs in $M$ such that $\alpha_1\cdots\alpha_n$ is a piecewise geodesic arc. By $[\cdot\alpha_1\cdots\alpha_n\cdot]$ we denote the geodesic arc that is homotopic to the piecewise geodesic arc with end points fixed. If the initial point of $\alpha_1$ and the terminal point of $\alpha_n$ are the same, by $[\alpha_1\cdots\alpha_n]$ we denote the corresponding geodesic closed curve. \par 
 
 The following lemma, which estimates the length of $[\cdot\alpha_1\cdots\alpha_n\cdot]$, was proved in \cite[Theorem 4.1 and Lemma 4.1]{KM-surface group} .
\begin{lemma} [Chain Lemma]
	There exists a universal constant $Q>0$ such that following holds. Suppose $\len(\alpha_i)\geq Q$ for any $i=1,\dots, n$ and 
	\[\Theta(t(\alpha_i),i(\alpha_{i+1}))\leq\delta<1\]
	 for any $i=1,\cdots,n-1$. Then the following estimates hold for some universal $O$
	\[
	\len\big([\cdot\alpha_1\cdots\alpha_n\cdot]\big)=\sum_i\len(\alpha_i)+O(n\delta)
	\]
	\[
	\Theta\left(i(\alpha_1),i\left([\cdot\alpha_1\cdots\alpha_n\cdot]\right)\right),\Theta\left(t(\alpha_n),t\left([\cdot\alpha_1\cdots\alpha_n\cdot]\right)\right)=O(n\delta)
	\]
	Moreover, assume that $\alpha_1\cdots\alpha_n$ is a piecewise geodesic curve and 
	\[
	\Theta(t(\alpha_n),i(\alpha_{1}))\leq\delta
	\]
	Then the following estimate holds
	\[
	\len\big([\alpha_1\cdots\alpha_n]\big)=\sum_i\len(\alpha_i)+O(n\delta)
	\]
\end{lemma}
The following lemma is a corollary of the Chain Lemma.
\begin{lemma}[Right Angle Chain]
	Let $Q$ be the constant from the Chain Lemma. Let $L\geq 2Q$. Suppose $\alpha_1,\dots,\alpha_n$ is a sequence of consecutive arcs such that 
	\begin{enumerate}
		\item for any $i=1,\dots,n$ we have $\len(\alpha_i)\geq L$
		\item for any $i=1,\dots,n-1$ we have 
		\[\left|\Theta\left(t(\alpha_i),i(\alpha_{i+1})\right)-\frac{\pi}{2}\right|\leq e^{-L}\]
	\end{enumerate}
	 Then the following estimates hold for some universal $O$
	\[
	\len\left([\cdot\alpha_1\cdots\alpha_n\cdot]\right)=\sum_i\len(\alpha_i)-(n-1)\log 2+O\left(ne^{-L}\right),
	\]
	\[
	\Theta\left(i(\alpha_1),i\left([\cdot\alpha_1\cdots\alpha_n\cdot]\right)\right),\Theta\left(t(\alpha_n),t\left([\cdot\alpha_1\cdots\alpha_n\cdot]\right)\right)=O\left(ne^{-L}\right).
	\] Moreover, assume that $\alpha_1\cdots\alpha_n$ is a piecewise geodesic curve, and
	\[\left|\Theta\left(t(\alpha_n),i(\alpha_1)\right)-\frac{\pi}{2}\right|\leq e^{-L},\]
	then the following estimate holds \[\len([\alpha_1\cdots\alpha_n])=\sum_i\len(\alpha_i)-n\log 2+O\left(ne^{-L}\right).\]
\end{lemma}
\begin{proof}
	Suppose $\eta_1,\eta_2$ are consecutive arcs such that $\len(\eta_i)\geq L/2$ for any $i=1,2$, and the following holds
	\[
	\left|\Theta\left(t(\eta_1),i(\eta_2)\right)-\frac\pi 2\right|\leq e^{-L}.
	\]
	Denote $\eta_3=[\cdot\eta_1\eta_2\cdot]$. Then by the hyperbolic trigonometry, we have
	\[
	\len(\eta_3)=\len(\eta_1)+\len(\eta_2)-\log 2+O\left(e^{-L}\right)
	\]
	\[
	\Theta\left(i(\eta_1),i(\eta_3)\right),\Theta\left(t(\eta_2),t(\eta_3)\right)=O\left(e^{-L}\right)
	\]
	for some universal $O$. Suppose $\alpha_1,\dots,\alpha_n$ is a sequence of consecutive arcs satisfying the conditions from the lemma. Let $a_i$ be the midpoint of $\alpha_i$ for any $i=1,\dots,n$, and $\beta_i$ be the geodesic arc connecting $a_i,a_{i+1}$ for any $i=1,\dots,n-1$. Let $\beta_0=\alpha_1|_{[0,a_1]}$ and $\beta_n=\alpha_n|_{[a_n,\len(\alpha_n)]}$. Then $\beta_0,\dots,\beta_n$ is a sequence of consecutive arcs satisfying the conditions from the Chain Lemma with $\delta=e^{-L}$. Thus the lemma follows from the Chain Lemma.
\end{proof}
\subsection{The theory of inefficiency}
In this subsection, we introduce the notion of semirandom maps, then we quote and revise lemmas in \cite[Subsection 4.2]{KM-Correction} for our purpose.
\begin{definition}
	Let $\alpha$ be an piecewise geodesic arc on a hyperbolic surface. By $\gamma$ we denote the geodesic arc with the same endpoints and homotopic to $\alpha$. We let $I(\alpha)=\len(\alpha)-\len(\gamma)$, and call $I(\alpha)$ the inefficiency of $\alpha$.
\end{definition}
\begin{lemma}\label{bounded distance}
	Let $\alpha$ denote an piecewise geodesic arc on $\mathbb H$, and let $\gamma$ be the geodesic arc with the same end points. Let $\pi:\alpha\rightarrow \gamma$ be the nearest point projection. Let 
	\[
	E(\alpha)=\sup_{x\in\alpha}d(x,\pi(x)). 
	\]
	Then $E(\alpha)\leq \frac 12I(\alpha)+\log 2$. 
\end{lemma}

\begin{lemma}\label{new-angle}
	There exist a universal constant $C$ such that the following holds. Let $\alpha\beta$ be an arc on the hyperbolic plane, where $\alpha$ is a piecewise geodesic arc and $\beta$ is a geodesic arc. Suppose $\gamma$ is the geodesic arc with the same endpoints as $\alpha\beta$. If $\len(\beta)\geq I(\alpha\beta)+1$ then the unoriented angle between $\gamma$ and $\beta$ is at most $Ce^{I(\alpha\beta)-\len(\beta)}$.
\end{lemma}
\begin{lemma}\label{bounded-ineff}
	Suppose that $\alpha\beta\gamma$ is a concatenation of three goedesic arcs in the hyperbolic plane, and let $\theta_{\alpha\beta}$ and $\theta_{\beta\gamma}$ be the two bending angles. Suppose $\theta_{\alpha\beta},\theta_{\beta\gamma}\leq 1.1$. Then $I(\alpha\beta\gamma)\leq D_1$ for some constant $D_1>0$.
\end{lemma}

\section{Preliminaries on the theory of semirandom map}
In this section, we review the theory of semirandom map. The randomness condition in the Good Pants Homology Theorem is equivalent to the statement that the map $\Phi$ is $(R\epsilon^{-1})^m$-semirandom. To verify that the constructed map $\Phi$ satisfies the randomness condition, we will estimate the semirandom norm at each step of the construction. 

\subsection{Definition of semirandom map}
For any set $X$, by $\mathbb RX$ we denote the vector space of finite formal sums of points in $X$. Let $f:X\rightarrow \mathbb RY$ be a map and $f(x)=\sum_yf_x(y)y$. For any measure $m$ on $X$, we define the measure $|f|_*m$ on $Y$ as follows
	\[
	|f|_*m(V)=\int_X\left(\sum_{y\in V}|f_x(y)|\right)dm(x).
	\]
	\begin{definition}
		Let $X,Y$ be two measurable spaces equipped with measure class $\Sigma_X,\Sigma_Y$ respectively. A map $f:X\rightarrow \mathbb RY$ is $K$-semirandom if for any measure $\sigma_X\in\Sigma_X$ there exists a measure $\sigma_Y\in\Sigma_Y$ such that 
		\[
		|f|_*\sigma_X\leq K\sigma_Y.\] 	
		The semirandom norm of a map $f$ is defined by
	 \[
	 ||f||_{s.r.}=\inf\,\{K:\text{f is $K$-semirandom }\}
	 \]
	\end{definition}
	
	We say that a class of measures is convex if it contains all convex combinations of its elements. The following two propositions are elementary. 
		
	\begin{proposition}
		If $f_i : X \rightarrow\mathbb RY$ is $K_i$-semirandom with respect
		to classes of measures $\Sigma_X$ and $\Sigma_Y$, $i = 1,2$, and if $\Sigma_Y$ is convex, then for $\lambda_i \in\mathbb R$, the map $(\lambda_1 f_1 + \lambda_2 f_2) : X \rightarrow \mathbb RY$ is $(\lambda_1 K_1 + \lambda_2 K_2)$-semirandom with respect to $\Sigma_X$ and $\Sigma_Y$.
	\end{proposition}
	\begin{proposition}
		Let $X,Y$ and $Z$ denote three spaces with classes of measures $\Sigma_X,\Sigma_Y,\Sigma_Z$ respectively. If $f:X\rightarrow\mathbb RY$ is $K$-semirandom with respect to $\Sigma_X$ and $\Sigma_Y$, and $g:Y\rightarrow\mathbb RZ$ is $L$-semirandom with respect to $\Sigma_Y$ and $\Sigma_Z$. Then $g\circ f:X\rightarrow\mathbb RZ$ is $KL$-semirandom with respect to $\Sigma_X$ and $\Sigma_Z$.
	\end{proposition}
	\begin{notation}[Partial Map]
		The maps in the subsequent sections are established via explicit construction. We make the following convention, which will be used frequently. Suppose $X,Y$ are spaces with measure classes $\Sigma_X,\Sigma_Y$ respectively. Consider a statement of the form: "For any element $x\in X$ satisfying condition A, there exists an element $f(x)\in\mathbb RY$ satisfying condition B." We say that this defines a partial map $f$ from $X$ to $\mathbb RY$, that is, the map is defined on the subset of $X$ where condition A holds. Any such partial map can be extended to a map defined on $X$ by setting the value zero for any elements that do not satisfy condition A. We define the semirandom norm of the partial map to be equal to the semirandom norm of the extended map.
	\end{notation}
	\subsection{Natural measure classes}
	In this paper, we consider the following spaces and their measure classes:
	\begin{enumerate}
		\item The space $\{1\}$ with the measure class containing the single measure $\sigma_1(1)=1$.
		\item The space of all geodesic arcs $Conn$ with the measure class containing the single measure $\sigma_C$ which is defined by setting
	\[
	\sigma_C(A)=e^{-\len(A)}.
	\]
	 Suppose $u,v$ are unit vectors.  We consider $Conn_{\epsilon,R}(u,v)$ as a subset of $Conn$ with the restricted measure $\sigma_C$.
	 \item By $G$ we denote the fundamental group of $M$ with base point $*$. The group element $A$ is identified with the geodesic arc homotopic to $A$ fixing the end points $*$. We consider the fundamental group $G$ with measure class containing the single measure $\sigma_G$ which is defined by setting 
	\[
	\sigma_G(A)=e^{-\len(A)}.\]
	 	\item The space of good curves $\Gamma_{\epsilon,R}$ with the measure class containing the single measure $\sigma_\Gamma$ which is defined by setting 
				\[\sigma_\Gamma(\gamma)=Re^{-2R}.\]
		\item The space of good pants $\PPi_{\epsilon,R}$ with the measure class containing the single measure $\sigma_{\PPi}$ which is defined by setting
		\[
		\sigma_{\PPi}(\Pi)=e^{-3R}.\]
		
	\end{enumerate}
		Note that the randomness condition from Good Pants Homology Theorem is equivalent to the statement that the map 
		\[\Phi:\Gamma_{\epsilon,R}\rightarrow\mathbb Q\PPi_{\epsilon,R}\]
		 is $(R\epsilon^{-1})^{m}$-semirandom with respect to the measure classes $\sigma_\Gamma$ and $\sigma_{\PPi}$.

		 \subsection{Standard maps are semirandom}
		   In this subsection, we estimate the semirandom norms of some standard maps.
		\begin{proposition}\label{forgetful}
			There exists a constant $C=C(M)$ such that the following holds. Let $u,v$ be two unit vectors.  For any $R>0$ the forgetful map 
			\[
			ft:Conn_{1,R}(u,v)\rightarrow\{1\}
			\]
			mapping each arc to the element $1$ is $C$-semirandom.
		\end{proposition}    
		\begin{proof}
			By the classical counting results \cite{Margulis}, there exists a constant $C=C(M)$ such that for any $u,v$ we have 
			\[
			\#Conn_{1,R}(u,v)\leq Ce^{R-1}
			\]
			Thus, we have the following estimate
			\begin{equation*}
			\begin{aligned}
				|ft|_*\sigma_C(1)=\sum e^{-\len(A)}&\leq\#Conn_{\epsilon,R}(u,v)e^{-R+1}\leq C=C\sigma_1(1)
			\end{aligned}
			\end{equation*}
			where the summation is taken over all $A\in Conn_{1,R}(u,v)$.
		\end{proof}
	\begin{proposition}\label{boundarysr}
	 	The boundary map 
	 	\[\partial :\PPi_{1,R}\rightarrow\mathbb Z\Gamma_{1,R}\] is $C$-semirandom map for some constant $C=C(M)$.
	 \end{proposition}
	 \begin{proof}
	 	Given $\gamma\in\Gamma_{1,R}$. According to the first property in Theorem \ref{equidistribution}, there are at most $CRe^R$ pair of pants that has $\gamma$ as a boundary for some constant $C=C(M)$. Thus, we have the following estimate
	 	\[
	 	|\partial |_*\sigma_{\PPi}(\gamma)=\sum_{\Pi}\sigma_{\PPi}(\Pi)\leq\frac{CRe^R}{e^{3R}}=C\sigma_\Gamma(\gamma)
	 	\]
	 	where the summation over all pair of pants that has a boundary $\gamma$. This implies that the boundary map is a $C$-semirandom map.
	 \end{proof} 
	 Recall that for any closed curve $\gamma$ and any third connection $\eta\in Conn_{\gamma}$, by $ Thd(\gamma,\eta)$ we denote the pair of pants constructed by the third connection construction. We consider the subset of $\Gamma\times Conn$ 
	\[\mathcal F\Gamma_{1,R}=\{(\gamma,\eta): \gamma\in\Gamma_{1,R},\eta\in Conn_\gamma, Thd(\gamma,\eta)\in\PPi_{1,R}\}\]
	equipped with the restricted measure $\sigma_\Gamma\times \sigma_C$.
	 \begin{proposition}\label{thirdsr}
	 	The third connection map 
	 	\[Thd:\mathcal F\Gamma_{1,R}\rightarrow\PPi_{1,R}\] is $CR$-semirandom with respect to the measure classes $\sigma_\Gamma\times\sigma_C$ and $\sigma_{\PPi}$ for some universal constant $C$.
	 \end{proposition}
	 \begin{proof}
	 	For any $\Pi\in\PPi_{1,R}$, let $\gamma_1,\gamma_2,\gamma_3$ be three boundaries of $\Pi$ and let $\eta_1,\eta_2,\eta_3$ be the geodesic third connection respectively. Since $\len(\eta_i)\geq R-C_1$ for some universal constant $C_1$, the following holds
	 	\[
	 	|Thd|_*\sigma_\Gamma\times\sigma_C(\Pi)=\sum_i\sigma_{\Gamma}(\gamma_i)\sigma_C(\eta_i)\leq 3e^{C_1}Re^{-3R}=3e^{C_1}R\text{ }\sigma_{\PPi}(\Pi)
	 	\]
	 	This implies that the third connection map is $3e^{C_1}R$-semirandom.
	 \end{proof}

		\subsection{The randomization method} 
		In this subsection, we develop the method of randomization. \par 
	 Suppose $X,Y,Z$ are spaces with measure classes $\Sigma_X,\Sigma_Y,\Sigma_Z$ respectively. Consider a map $f^*:X\times Y\rightarrow \mathbb RZ$ where $X\times Y$ is endowed with the product measure classes $\Sigma_X\times\Sigma_Y$. Given an element $\underline B_Y\in\mathbb R Y$. We define the ramdomization of $f^*$ with respect to $\underline B_Y$ as the map $f:X\rightarrow\mathbb RZ$ given by$f(x)=f^*(x,\underline B_Y)$. The semirandom norm of the randomized map $f$ can be estimated by 
	\[
	||f||_{s.r.}\leq ||f^*||_{s.r.}\cdot||\underline B_Y||_{s.r.},
	\]
	where the semirandom norm of an element is defined as follows.
	 \begin{definition}[Semirandom Norm of Element]
	  Let $Y$ be a space with measure classes $\Sigma_Y$. An element $B\in\mathbb R Y$ is called $K$-semirandom if the map $
	 \{1\}\rightarrow\mathbb RY$
	 mapping $1$ to the element $B$ is $K$-semirandom with respect to the measure classes $\sigma_1$ and $\Sigma_Y$. By $||B||_{s.r.}$ we denote the semirandom norm of this map.
	 \end{definition} 
	 In particular, when $f:X\times Y\rightarrow \mathbb Z Z$ is an integer-valued map and $\underline B_Y\in \frac{\mathbb Z}{N} Y$ has an integer denominator $N$. The coefficients of the randomized map are effectively controlled by
	 \[
	 f:X\rightarrow \frac{\mathbb Z}{N}Z.
	 \] The following proposition establishes the existence of a random element, whose coefficients have effectively controlled denominators.  
	 
	\begin{lemma}[Random Element]\label{randomization}
		Let $X=Conn_{\epsilon,L}(u,v)$. Suppose  $2R\geq L\geq L(\epsilon)$, where 
	\[L(\epsilon)=\max\{-q_0^{-1}\log\epsilon,-\log\epsilon,L_0\}\]
	is the constant defined on the Connection Lemma. Then there exists an element $\underline B_X\in\frac{\mathbb Z^+}{\lfloor e^{2R}\rfloor}X$ such that
	  \[
	  \sum_{x\in X}\underline B_X(x)=1,
	  \] 
	  and $||\underline B_X||_{s.r.}\leq C\epsilon^{-3}$, for some constant $C$ depending only on $M$. Such an element is called a $(\epsilon,R)$-effectively random element in $X$.
	\end{lemma}
	 			
	\begin{proof}
		Let $\underline B_X\in\frac{\mathbb Z^+}{\lfloor e^{2R}\rfloor}X$ be an element almost averaging the weight such that 
		\[
		\sum_{x\in X}\underline B_X(x)=1,
		\]
		and 
		\begin{equation*}
			\begin{aligned}
				\left|\underline B_X(x)-\frac{1}{\#X}\right|\leq \frac{1}{\lfloor e^{2R}\rfloor}&\text{. Thus, }&\left|\underline B_X(x)\right|\leq \frac{1}{\# X}+\frac{1}{\lfloor e^{2R}\rfloor}.
			\end{aligned}
		\end{equation*}
		By the classical counting result \cite{Margulis}, there are at most $C_1e^L$ many geodesic arcs in $Conn_{\epsilon,L}(u,v)$ for some constant $C_1=C_1(M)$. Thus $C_1^{-1}\#X\leq e^L\leq e^{2R}$. Combining this with the above estimate, we get
		\[
		\left|\underline B_X(x)\right|\leq \frac{1}{\#X}+\frac{1}{C_1^{-1}\#X}\leq  \frac{1+C_1}{C_2}\epsilon^{-3}e^{-L},
		\]
		where $C_2=C_2(M)$ is the constant from the Connection Lemma. Let $C=\frac{1+C_1}{C_2}$. By definition $||\underline B_X||_{s.r.}\leq C\epsilon^{-3}$. This completes the proof.
	\end{proof}

	\section{The replacement map and the itemization theorems} 
	In this section, we establish the most technical part in the proof of the Good Pants Homology Theorem, which is structurally modeled on (but simplified) the corresponding constructions in \cite[Section 5-7]{KM-Correction}.\par 
	We define the replacement map in this section. In brief, letting $u\in T^1M$ be a unit vector in $M$, the replacement map $\mathcal R$ maps a suitable element $A\in Conn_{\epsilon^2,<\infty}(-u,u)$ to a multi-curve
	\[
		\mathcal R(A)=\frac 12\left([A\underline B_A]-[\bar A\underline B_A]\right),
		\]
	where $\underline B_A$ is a suitable random element in $Conn_{\epsilon^2,<\infty}(u,-u)$, such that $\mathcal R(A)$ is a good multi-curve and has the same homology class as $A$. Similarly, we can replace $B\in Conn_{\epsilon^2,<\infty}(u,-u)$ with a multi-curve $\mathcal R(B)$.\par 
	Next, in the Two-Part Itemization Theorem for any suitable elements $A\in Conn_{\epsilon^2,<\infty}(-u,u)$ and $B\in Conn_{\epsilon
	^2,<\infty}(u,-u)$ such that $[AB]$ is a good curve, we construct a good multi-pants $Item_2$ such that
	\begin{equation*}
			\begin{aligned}
				\text{ }[AB]=\partial Item_2(A,B)+\mathcal R(A)+\mathcal R(B).
			\end{aligned}
		\end{equation*}
	And we want to go one step further and prove that for any suitable $A_1,A_2,B_1,B_2$, there exists a good multi-pants $Item_4$ such that
		\begin{equation*}
		\begin{aligned}
			\text{ }[A_1B_1A_2B_2]=\partial  &Item_4(A_1,A_2,B_1,B_2)+\\
			&\mathcal{R}(A_1)+\mathcal{R}(A_2)+\mathcal{R}(B_1)+\mathcal{R}(B_2),
		\end{aligned}
	\end{equation*}
	which is established in the Four-Part Itemization Theorem.
	
	\subsection{Narrow connections and the replacement map}
	In this subsection, we construct the replacement map. The replacement map is defined on the set of narrow connections. We first introduce the notion of narrow connection.
	\begin{definition}[Narrow connection set]
		Given $\epsilon,R>0$. Let $u$ be a unit vector on $M$. We define
		\[
		LN_{\epsilon,R}(u)=\{A\in Conn_{\epsilon^2,<\infty}(-u,u):2R-L(\epsilon^2)\geq \len(A)\geq L(\epsilon^2)\},
		\]
		where $L(\epsilon^2)$ is the constant in the Random Element Lemma. The set $LN_{\epsilon,R}(u)$ is called the left narrow connection set with respect to $u$. A connection $A\in LN_{\epsilon,R}(u)$ is said to be a left narrow geodesic arc with respect to $u$. Similarly, we define
		\[
		RN_{\epsilon,R}(u)=\{B\in Conn_{\epsilon^2,<\infty}(u,-u):2R-L(\epsilon^2)\geq \len(B)\geq L(\epsilon^2)\},
		\]
		called the right narrow connection set with respect to $u$.
	\end{definition}
	 We are now going to construct the replacement map $\mathcal R$. For any $A\in LN_{\epsilon,R}(u)$, let 
	\[
		\mathcal F(A)=Conn_{\epsilon^2,2R-\len(A)}(u,-u),\]
	and let $\underline B_A$ be a $(\epsilon^2,R)$-effectively random element in $\mathcal F(A)$, whose existence is guaranteed by the length assumption of $A$ and the Random Element Lemma. Similarly, for any $B\in RN_{\epsilon,R}(u)$, let 
	\[
		\mathcal F(B)=Conn_{\epsilon^2,2R-\len(B)}(-u,u),\]
	and let $\underline A_B$ be a $(\epsilon^2,R)$-effectively random element in $\mathcal F(B)$. Note that 
	\begin{equation*}
		\begin{aligned}
			\mathcal F(A)=\mathcal F(\overline A)&\text{, and }&\mathcal F(B)=\mathcal F(\bar B).
		\end{aligned}
	\end{equation*}
	Therefore, we can take $\underline B_A=\underline B_{\overline A}$ and $\underline A_B=\underline A_{\bar B}$.
	\begin{definition}[Replacement map]
		 Given constants $\epsilon,R>0$. For any $A\in LN_{\epsilon,R}(u)$, we define 
		 \[
		\mathcal R(A)=\frac 12\left([A\underline B_A]-[\bar A\underline B_A]\right).
		\]
		Similarly, for any $B\in RN_{\epsilon,R}(u)$ we define 
		\[\mathcal R(B)=\frac 12\left([\underline A_BB]-[\underline A_B\bar B]\right).\] 
	\end{definition}
	 \begin{proposition}\label{replace-inverse}
	 	There exist constants $\epsilon_1,R_1>0$ only depending on $M$ such that for any $R\geq R_1$ and $\epsilon_1\geq \epsilon$ the following holds. For any $A\in LN_{\epsilon,R}(u)$, the multi-curve 
	 	\[
	 	\mathcal R(A)\in\frac{\mathbb Z\Gamma_{\epsilon,R}}{2\lfloor e^{2R}\rfloor},
	 	\]
	 	and 
	 	\[
	 	\mathcal R(A)=-\mathcal R(\bar A).
	 	\]
	 \end{proposition}
	 \begin{proof}
	 	The coefficients of $\mathcal R(A)$ belong to $(2\lfloor e^{2R}\rfloor)^{-1}\mathbb Z$, because the coeffecients of $\underline B_A$ belong to $(\lfloor e^{2R}\rfloor)^{-1}\mathbb Z$. Since $\underline B_A=\underline B_{\bar A}$, we have
	 	\begin{equation*}
	 		\begin{aligned}
	 			\mathcal R(A)=\frac 12\left([A\underline B_{\bar A}]-[\bar A\underline B_{\bar A}]\right)=-\frac 12\left([\bar A\underline B_{\bar A}]-[A\underline B_{\bar A}]\right)=-\mathcal R(\bar A).
	 		\end{aligned}
	 	\end{equation*}
	 	It suffices to prove that for any $B\in \mathcal F(A)$ the curve $[AB]\in\Gamma_{\epsilon,R}$. This is then established by the Chain Lemma. 
	 \end{proof}
	The following proposition is a preliminary to the lemma in the subsequent subsection. It also implies the semirandomness of the replacement map, although this fact is not needed for our main results. We prove the proposition here and omit the verification of semirandomness.
	\begin{proposition}\label{8}
			There exists a constant $C=C(M)$ such that the following holds. For any closed geodesic curve $\gamma\in\Gamma_{1,R}$, there exist at most $CR^2$ pairs of elements\[(A,B)\in Conn_{1,<\infty}(-u,u)\times Conn_{1,<\infty}(u,-u)\]
		 such that $\gamma=[AB]$.
		\end{proposition}
		\begin{proof}
		For any curves $\gamma\in \Gamma_{\epsilon,R}$, let $g_\gamma\in\pi_1(M)$ be a group element homotopy to $\gamma$. Let $\mathbb T_\gamma\rightarrow M$ be a covering with respect to the subgroup $\langle g_\gamma\rangle<\pi_1(M)$ such that any piecewise geodesic curve $AB$ satisfying $[AB]=\gamma$ can be lifted to the cover $\mathbb T_\gamma$. Vice versa, the pair of arcs $(A,B)$ is fully determined by two corner points of the lifted piecewise geodesic in $\mathbb T_\gamma$, denoted by $c_1,c_2$. Let $\gamma\subset\mathbb T_\gamma$ be the lifting of $\gamma$ itself. According to the Lemma \ref{bounded distance}, the two corner points lie in the distance $K$ neighborhood of $\gamma$ for some universal constant $K$. We denote the neighborhood by $\mathbb T_{\gamma,K}$. By applying Fubini Theorem, we obtain
		\[
		Vol(\mathbb T_{\gamma,K})=C(K)\len(\gamma)
		\] 
		Let $(A^\prime,B^\prime)$ be the other pair of geodesic arcs satisfying $[A^\prime B^\prime]=\gamma$. We denote two lifted corner points of $A^\prime B^\prime$ by $c_1^\prime,c_2^\prime$. Assume $c_i$ and $c_i^\prime$ are $I/2$-closed for any $i=1,2$, where $I$ is the injective radius of $M$. Then we have $A^\prime=A$ and $B^\prime=B$. Therefore, we have
		\[
		\#\{(A,B):\gamma=[AB]\}\leq\left(\frac{Vol(\mathbb T_{\gamma,K})}{Vol(\text{radius $I$ ball})}\right)^2=CR^2
		\]
		for some $C=C(M)$. This completes the proof.
		\end{proof}
	\par
	
		\subsection{The Square Lemma}
	This subsection contains a technical lemma, which solves a special case of the Good Pants Homology Theorem, called the Square Lemma.\par 
	 Before stating the Square Lemma, we recall the $\boxtimes$ operator in measure theory. Let $X_1$ be a space equipped with a measure $\sigma_1$ and $X_2$ be a space equipped with a measure $\sigma_2$. By $\sigma_1\boxtimes\sigma_2$, we denote the measure class on $X_1\times X_2$ containing all measures $\sigma$ satisfies that $(\pi_i)_*\sigma\leq \sigma_i$ for any $i=1,2$ where $\pi_i$ is the projection map from $X_1\times X_2$ to $X_i$.
	\begin{lemma}[Square Lemma]\label{square}
	There exist constants $\epsilon_1,R_1,q_1>0$ depending only on $M$ such that for any $R\geq R_1$ and $\epsilon_1 \geq\epsilon\geq e^{-q_1R}$ the following holds. Let $u$ be a unit vector in $M$.  For any 
	\[
	A_1,A_2\in Conn_{\epsilon^2,L_a}(-u,u)
	\]
	\[
	B_1,B_2\in Conn_{\epsilon^2,L_b}(u,-u)
	\]
	satisfying $L_a+L_b=2R$ and $L_a,L_b\geq 2L(\epsilon^2)$, there exists a multi-pants 
	\[
	Sq(A_1,A_2,B_1,B_2)\in\frac{\mathbb Z}{\lfloor e^{2R}\rfloor^4}\PPi_{\epsilon,R}
	\]
	 such that 
	\[
	\partial Sq(A_1,A_2,B_1,B_2)=\sum_{i,j=1,2}(-1)^{i+j}[A_iB_j],
	\]
	and the partial map $Sq$ is $(R\epsilon^{-1})^{24}$-semirandom with respect to the measure classes $\sigma_C^{\boxtimes 2}(A_1,A_2)\times\sigma_C^{\boxtimes 2}(B_1,B_2)$ and $\sigma_{\PPi}$.
	\end{lemma}
	\begin{proof}
		 For any 
		 \begin{equation*}
		 	\begin{aligned}
		 		A\in Conn_{\epsilon^2,L_a}(-u,u)&\text{, and }&B\in Conn_{\epsilon^2,L_b}(u,-u),
		 	\end{aligned}
		 \end{equation*}
		\textbf{Claim:} There exists a multi-pants depending only on $A$ and $B$
		\[P^*=P^*(A,B)\in \frac{\mathbb Z}{\lfloor e^{2R}\rfloor^4}\PPi_{\epsilon,R},\] 
		a multi-curve $S^*_a$ depending only on $A$, and a multi-curve $S^*_b$ depending only on $B$ such that
		 \begin{equation}\label{E1.7}
		 	\partial  P^*=[AB]-S^*_a-S^*_b.
		 \end{equation}Suppose $A_1,A_2$ and $B_1,B_2$ satisfy the condition in the lemma. Let
	\[
	Sq(A_1,A_2,B_1,B_2)=\sum_{i,j}(-1)^{i+j}P^*(A_i,B_j)
	\]
	We apply the equation (\ref{E1.7}) to get that the multi-pants $Sq$ satisfies the needed property.
	\begin{proof}[\textbf{Proof of the claim}]
		 Given $A,B$. We are going to construct $P^*,S_a^*$ and $S_b^*$. Let $A^-,A^+$ be two segments of $A$ minus the midpoint of $A$ such that $A=A^-A^+$. Let $B^-,B^+$ be two segments of $B$ minus the midpoint of $B$ such that $B=B^-B^+$. Let $v_0$ be an auxiliary unit vector. We denote
		\begin{equation*}
		\begin{aligned}
			\mathcal C(A)=& Conn_{\epsilon^2,L_b/2+\log 2}\left(\sqrt{-1}t(A^-),v_0\right)\\
			\mathcal C(B)=& Conn_{\epsilon^2,L_a/2+\log 2}\left(\sqrt{-1}t(B^-),-v_0\right)\\
			\mathcal W_1=& Conn_{\epsilon^2,R+2\log 2}\left(-\sqrt{-1}u,\sqrt{-1}v_0\right)\\
			\mathcal W_2=& Conn_{\epsilon^2,R+2\log 2}\left(\sqrt{-1}u,-\sqrt{-1}v_0\right)
		\end{aligned}
		\end{equation*}
		 For any $Z_a\in\mathcal C(A)$ and $Z_b\in\mathcal C(B)$, the piecewise geodesic arc $Z_a\overline Z_b$ is homotopic to a third connection of the curve $AB$. By $P_1(A,B,Z_a,Z_b)$ we denote the pair of pants constructed by the third connection construction, such that the other two boundaries of $P_1$ are
		\begin{equation*}
			\begin{aligned}
				\eta_1=[B^+A^-Z_a\overline Z_b]&\text{, and }&\eta_2=[Z_b\overline Z_aA^+B^-]
			\end{aligned}
		\end{equation*}
		Meanwhile, for any $W_1\in\mathcal W_1$ and $W_2\in\mathcal W_2$, let $P_2(A,B,Z_a,Z_b,W_1,W_2)$ be the summation of two pairs of pants constructed by adding the third connection $W_1$ to $\eta_1$ and adding the third connection $W_2$ to $\eta_2$ respectively. By adding the third connection $W_1$ to the curve $\eta_1$, we get a pair of pants with two other boundaries
	\begin{equation*}
		\begin{aligned}
			\eta_3=\left[A^-Z_a\overline{W_1}\right]&\text{, and }&\eta_4=\left[W_1\overline{Z_b} B^+\right]
		\end{aligned}
	\end{equation*}
	By adding the third connection $W_2$ to the curve $\eta_2$, we get a pair of pants with two other boundaries
	\begin{equation*}
		\begin{aligned}
			\eta_5=\left[\overline{Z_a} A^+W_2\right]&\text{, and }&\eta_6=\left[\overline{W_2}B^-Z_b\right]
		\end{aligned}
	\end{equation*}
	Let $S_a=\eta_3+\eta_5$, $S_b=\eta_4+\eta_6$, and let $P=P_1+P_2$. Therefore, for any $Z_a,Z_b,W_1,W_2$, we obtain $
	\partial  P=[AB]-S_a-S_b$. The lengths of geodesic curves $\eta_i$ can be estimated by the Right Angle Chain Lemma. There exists a universal constant $\epsilon_0$ such that if $\epsilon_0\geq \epsilon$ then $\eta_i\in\Gamma_{\epsilon,R}$ for any $i=1,\dots,6$. This implies that $P\in\mathbb Z\PPi_{\epsilon,R}$.
	\par Since the assumptions about $L_a,L_b$ and $\epsilon,R$,  the sets $\mathcal C(A),\mathcal C(B),\mathcal W_1,$ $\mathcal W_2$ satisfy the condition from the Random Element Lemma. Let $\underline Z_a,\underline Z_b,\underline W_1,\underline W_2$ be a $(\epsilon^2,R)$-effectively random element in the corresponding connection sets, respectively. We define 
	\[
	P^*(A,B)=P(A,B,\underline Z_a,\underline Z_b,\underline W_1,\underline W_2)\in\frac{\mathbb Z}{\lfloor e^{2R}\rfloor^4}\PPi_{C\epsilon^2,R},
	\]
	$S_a^*(A)=S_a(A,\underline Z_a,\underline W_1,\underline W_2)$, and $S_b^*(B)=S_b(B,\underline Z_b,\underline W_1,\underline W_2)$. 
	\end{proof}
	\textbf{Estimate the semirandom norm of $Sq$.} 
	Let $\mathcal C(A),\mathcal C(B)$ and $\mathcal W_1,\mathcal W_2$ be the measure spaces equipped with the measure $\sigma_C$. By definition, for any $\Pi\in\PPi_{1,R}$ we have the following estimates
	\begin{equation*}
		\begin{aligned}
			|P_1|_*(\sigma_C^{\times 4})(\Pi)= \sum_{A,B,Z_a,Z_b}e^{-\len(A)-\len(B)-\len(Z_a)-\len(Z_b)}\leq C^\prime\sum_{A,B,Z_a,Z_b} e^{-3R}
		\end{aligned}
	\end{equation*}
	where the summation is taken over all elements $(A,B,Z_a,Z_b)$ satisfying 
	\[P_1(A,B,Z_a,Z_b)=\Pi,\] and $C^\prime$ is some universal constant.  According to Proposition \ref{8}, there exists a constant $C^{\prime\prime}=C^{\prime\prime}(M)$ such that there are at most $C^{\prime\prime}R^2$ many pairs of elements $(A,B)$ such that the curve $[AB]$ is a boundary of the pair of pants $\Pi$. This further implies that there are at most $C^{\prime\prime}R^2$ four-tuples of elements $(A,B,Z_a,Z_b)$ satisfying $P_1(A,B,Z_a,Z_b)=\Pi$, since the third connection map is injective. Therefore, letting $C_1=C^\prime C^{\prime\prime}$, the map $P_1$ is $C_1R^2$-semirandom. 
	\par The pants $P_2$ is constructed by first taking the boundary of $P_1$, then adding the third connections. Therefore, the semirandom norm of $P_2$ is bounded by the semirandom norm of the map $P_1$, the boundary map, and the third connection map. According to Proposition \ref{thirdsr} and Proposition \ref{boundarysr}, we get that the map $P$ is $C_2R^3$-semirandom for some constant $C_2=C_2(M)$. According to the Random Element Lemma, the semirandom norms of the random elements is bounded by $\epsilon^{-6}$ such that the map $P^*$ is a $C_3R^3\epsilon^{-24}$-semirandom map for some cosntant $C_3=C_3(M)$. 
	\par 
	The map $Sq$ is the composition of the map
	 \[(A_1,A_2,B_1,B_2)\mapsto\sum_{ij}(-1)^{i+j}(A_i,B_j)\]
	  and the map $P^*$. Since the former map is $4$-semirandom with respect to the measure classes $\sigma_C^{\boxtimes 2}\times\sigma_C^{\boxtimes 2}$ and $\sigma_C\times\sigma_C$, we obtain that the map $Sq$ is $4C_3R^3\epsilon^{-24}$-semirandom. If $R\geq 4C_2$, then the map $Sq$ is $R^4\epsilon^{-24}$-semirandom. 
	\end{proof}
		\subsection{The Two-Part Itemization Theorem}

		\begin{theorem}[Two-Part Itemization Theorem]
	There exist constants $\epsilon_1,R_1,q_1>0$ depending only on $M$ such that for any $R\geq R_1$ and $\epsilon_1\geq\epsilon\geq e^{-q_1R}$ the following holds. Suppose $A\in LN_{\epsilon,R}(u)$ and $B\in RN_{\epsilon,R}(u)$ such that  $[AB]\in\Gamma_{\epsilon,R}$. Then there exists a multi-pants
		\[
		Item_2(A,B)\in\frac{\mathbb Z\PPi_{\epsilon,R}}{2\lfloor\exp(2R)\rfloor^5}
		\]
		such that the following holds
		\begin{equation*}
			\begin{aligned}
				\text{ }[AB]=\partial Item_2(A,B)+\mathcal R(A)+\mathcal R(B),
			\end{aligned}
		\end{equation*}
		and the partial map $Item_2$ is $(R\epsilon^{-1})^{25}$-semirandom with respect to the measure classes $\sigma_C^2$ and $\sigma_{\PPi}$ for some $m=m(M)>0$.
	\end{theorem}
	\begin{proof}
		Since $\len(A),\len(B)\geq L(\epsilon^2)$ and $\len(A)+\len(B)\approx 2R$, the replacement map is well-defined on the elements $A$ and $B$. By applying the Square Lemma, we obtain the following two equations
		\begin{equation*}
			\begin{aligned}
				\partial Sq(A,\bar A,B,\underline B_A)&=\left([AB]-[\bar AB]\right)-\left([A\underline B_A]-[\bar A\underline B_A]\right)\\
				\partial Sq(\bar A,\underline A_B,B,\bar B)&=\left([\bar AB]-[\bar A\bar B]\right)-\left([\underline A_BB]-[\underline A_B\bar B]\right)
			\end{aligned}
		\end{equation*}
		Let
		\[Item_2(A,B)=\frac 12Sq(A,\bar A,B,\underline B_A)+\frac 12Sq(\bar A,\underline A_B,B,\bar B).\] 
		Combining the above two equations, we obtain
		\[
		\partial Item_2(A,B)=[AB]-\mathcal R(A)-\mathcal R(B).
		\]
		The semirandomness of the partial map $Item_2$ follows from the semirandomness of $Sq$.
	\end{proof}

		\subsection{The Exchange Lemma}
		In this subsection, we prove the Exchange Lemma, which is a preliminary lemma to the Four-Part Itemization Theorem. The Exchange Lemma is proved by iteratively using the Square Lemma.
		\begin{lemma}[Exchange Lemma]
			There exist constants $\epsilon_1,R_1,q_1>0$ depending only on $M$ such that for any $R\geq R_1$ and $\epsilon_1\geq\epsilon\geq e^{-q_1R}$ the following holds. Suppose $A_1,A_2\in LN_{\epsilon,R}(u)$ and $B_1,B_2\in RN_{\epsilon,R}(u)$ such that $[A_1B_1A_2B_2]\in\Gamma_{\epsilon,R}$. Then there exists a multi-pants
			\[
			Exch(A_1,A_2,B_1,B_2)\in \frac{\mathbb Z\PPi_{\epsilon,R}}{2\lfloor e^{2R}\rfloor^4}
			\]
			such that 
			\[
			\partial  Exch=[A_1B_1A_2B_2]-[A_1B_2A_2B_1],
			\]
			and the partial map $Exch$ is $(R\epsilon^{-1})^{26}$-semirandom with respect to the measure classes $\sigma_C^{\times 4}$ and $\sigma_{\PPi}$.
		\end{lemma}
		\begin{proof}
	 	 We first fix some notations. Given $A_1,A_2\in LN_{\epsilon,R}(u)$. For any two geodesic arcs $C,D\in  RN_{\epsilon,R}(u)$,  we denote 
		\[\{C,D\}=[A_1CA_2D]-[A_1DA_2C].\]
		Two pairs of right narrow connections $(C_1,D_1)$ and $(C_2,D_2)$ are called adjacent if for any $i\in\{1,2\}$ and $j\in\{1,2\}$ the following holds
		\[
		[A_1C_iA_2D_j],[A_1D_jA_2C_i]\text{ }\in \Gamma_{\epsilon,R}.
		\]
		Suppose $(C_1,D_1),(C_2,D_2)$ are adjacent. Then, by applying the Square Lemma, we get the following two equations
		\begin{equation*}
			\begin{aligned}
				\partial Sq(A_1C_1A_2,A_2C_1A_1,D_1,D_2)&=\{C_1,D_1\}-\{C_1,D_2\}\\
				\partial Sq(A_2D_2A_1,A_1D_2A_2,C_1,C_2)&=\{C_1,D_2\}-\{C_2,D_2\}
			\end{aligned}
		\end{equation*}
		Letting
		\begin{equation*}
			\begin{aligned}
				P(A_1,A_2,C_1,C_2,D_1,D_2)&=Sq(A_1C_1A_2,A_2C_1A_1,D_1,D_2)\\
				&+Sq(A_2D_2A_1,A_1D_2A_2,C_1,C_2),
			\end{aligned}
		\end{equation*}
		we then obtain 
		\begin{equation}\label{E1.8}
			\partial  P(A_1,A_2,C_1,C_2,D_1,D_2)=\{C_1,D_1\}-\{C_2,D_2\}.
		\end{equation}
		
		\textbf{The construction of $Exch$.} Suppose $A_1,A_2,B_1,B_2$ are elements that satisfy the conditions from the lemma. Denote $L_1=\len(B_1),L_2=\len(B_2)$. Let $n=\lfloor R\epsilon^{-2}\rfloor $ and $\delta=(L_1-L_2)/n$. For any $k=1,\dots,n-1$, we take a pair of geodesic arcs 
		\[
		(C_k,D_k)\in Conn_{\epsilon^2,L_1+k\delta}(u,-u)\times Conn_{\epsilon^2,L_2-k\delta}(u,-u),
		\] 
		and let $(C_0,D_0)=(B_1,B_2),(C_n,D_n)=(B_2,B_1)$. Therefore, for any $k=0,\cdots,n-1$, two pairs of geodesic arcs $(C_k,D_k)$ and $(C_{k+1},D_{k+1})$ are adjacent. Let
		\[
		Exch(A_1,A_2,B_1,B_2)=\frac 12\sum_{k=0}^{n-1}P(A_1,A_2,C_k,C_{k+1},D_k,D_{k+1}).
		\]
		By applying Equation (\ref{E1.8}), we obtain
		\begin{equation*}
			\begin{aligned}
				\partial Exch(A_1,A_2,B_1,B_2)&=\frac 12(\{C_0,D_0\}-\{C_n,D_n\})=\{B_1,B_2\}\\
				&=[A_1B_1A_2B_2]-[A_1B_2A_2B_1].
			\end{aligned}
		\end{equation*}
		Therefore, the multi-pants $Exch$ satisfies the needed property.\par 
		\textbf{Estimating the semirandom norm of $Exch$.}  
		According to the Square Lemma, the map $P$ is $(R\epsilon^{-1})^{24}$-semirandom with respect to the measure classes $\sigma_C^{\times 2}(A)\times \sigma_C^{\boxtimes 2}(C)\times \sigma_C^{\boxtimes 2}(D)$ and $\sigma_{\PPi}$. Meanwhile the map 
		\[
		(B_1,B_2)\mapsto (C_k,C_{k+1},D_k,D_{k+1})
		\]
		is $4$-semirandom with respect to the measure classes $\sigma_C^{\times 2}(B)$ and $\sigma_C^{\boxtimes 2}(C)\times \sigma_C^{\boxtimes 2}(D)$. Therefore, we have 
		\[
		||Exch||_{s.r.}\leq \sum_{k=0}^{n-1}||P(A_1,A_2,C_k,C_{k+1},D_k,D_{k+1})||_{s.r.}\leq 4n||P||_{s.r.}
		\]
		Since $n\approx R\epsilon^{-2}$, the map $Exch$ is $(R\epsilon^{-1})^{26}$-semirandom.
		\end{proof}

		 \subsection{The Four-Part Itemization Theorem}
		 \begin{theorem}[Four-Part Itemization Theorem]
	There exist constants $\epsilon_1,R_1,q_1>0$ depending only on $M$ such that for any $R\geq R_1$ and $\epsilon_1\geq\epsilon\geq e^{-q_1R}$ the following holds.  Suppose $A_1,A_2\in LN_{\epsilon,R}(u)$ and $B_1,B_2\in RN_{\epsilon,R}(u)$ such that $[A_1B_1A_2B_2]\in\Gamma_{\epsilon,R}$. Then there exists a multi-pants
		 \[Item_4(A_1,A_2,B_1,B_2)\in \frac{\mathbb Z\PPi_{\epsilon,R}}{2\lfloor\exp(2R)\rfloor^5}\] 
		  such that the following holds
	\begin{equation*}
		\begin{aligned}
			\text{ }[A_1B_1A_2B_2]=\partial  &Item_4(A_1,A_2,B_1,B_2)+\\
			&\mathcal{R}(A_1)+\mathcal{R}(A_2)+\mathcal{R}(B_1)+\mathcal{R}(B_2),
		\end{aligned}
	\end{equation*}
		and the partial map $Item_4$ is $(R\epsilon^{-1})^{30}$-semirandom with respect to the measure classes $\sigma_C^{\times 4}$ and $\sigma_{\PPi}$.
		
	\end{theorem}
	\begin{proof}
		The construction of the multi-pants $Item$ proceeds as follows.
	 \begin{equation*}
			\begin{aligned}
				2Item(A_1,A_2,B_1,B_2)&=Sq\left(A_1,\bar  {A_1},\text{ }B_1A_2B_2,\text{ }\underline B_{A_1}\right)\\
				&+Sq\left(A_2 B_2\bar  {A_1},\underline A_{B_1},B_1,\bar  B_1\right)\\
				&+Sq\left(A_2,\bar  {A_2},\text{ }B_2\bar  {A_1}\text{ }\bar  B_1,\text{ }\underline B_{A_2}\right)\\
				&+Sq\left(\bar {A_1}\text{ }\bar {B}_1\bar {A_2},\text{ }\underline A_{B_2},B_2,\bar  B_2\right)\\
				&-Exch(A_1,A_2,B_1,B_2)
			\end{aligned}
		\end{equation*}
	 By applying the Square Lemma, the following four equations hold
		\begin{equation*}
			\begin{aligned}
				\partial Sq\left(A_1,\bar  {A_1},\text{ }B_1A_2B_2,\text{ }\underline B_{A_1}\right)&=[A_1B_1A_2B_2]-[\bar {A_1}B_1A_2B_2]-2\mathcal{R}(A_1)\\
				\partial Sq\left(A_2 B_2\bar  {A_1},\underline A_{B_1},B_1,\bar  B_1\right)&=[\bar {A_1}B_1A_2B_2]-[\bar {A_1}\bar {B_1}A_2B_2]-2\mathcal{R}(B_1)\\
				\partial Sq\left(A_2,\bar  {A_2},\text{ }B_2\bar  {A_1}\text{ }\bar  B_1,\text{ }\underline B_{A_2}\right)&=[\bar {A_1}\bar {B_1}A_2B_2]-[\bar {A_1}\bar {B_1}\text{ }\bar {A_2}B_2]-2\mathcal{R}(A_2)\\
				\partial Sq\left(\bar {A_1}\text{ }\bar {B}_1\bar {A_2},\text{ }\underline A_{B_2},B_2,\bar  B_2\right)&=[\bar {A_1}\bar {B_1}\text{ }\bar {A_2}B_2]-[\bar {A_1}\bar {B_1}\text{ }\bar {A_2}\bar {B_2}]-2\mathcal{R}(B_2)
			\end{aligned}
		\end{equation*}
		By applying the Exchange Lemma, the following equation holds
		\[
			\partial  Exch=[A_1B_1A_2B_2]-[A_1B_2A_2B_1].
			\]
		Since 
		\[\left[\bar {A_1}\bar {B_1}\text{ }\bar {A_2}\bar {B_2}\right]=-\left[A_1B_2A_2B_1\right],\]
		by taking the summation of the above six equations, we obtain that the multi-pants $Item$ satisfies 
		\begin{equation*}
		\begin{aligned}
			\partial  &Item(A_1,A_2,B_1,B_2)=\\
			&[A_1B_1A_2B_2]-\left(\mathcal{R}(A_1)+\mathcal{R}(A_2)+\mathcal{R}(B_1)+\mathcal{R}(B_2)\right)	
		\end{aligned}
	\end{equation*} 
	 The semirandom norm of the map $Item$ is bounded by the semirandom norms of the map $Exch$ and the map $Sq$. This implies that the map $Item$ is $(R\epsilon^{-1})^{30}$-semiranodm.
	\end{proof}

\section{From the group homology to good pants}
This section proceeds as follows. In Subsections 6.1 and 6.2, we introduce the notion of narrow triangles and prove the Narrow Triangle Replacement Lemma. Technically, these two subsections are a continuation of the previous section. Then, in Subsections 6.3 and 6.4, we introduce the notion of bounded group homology and prove the Group-to-Pants Theorem.
\par  In brief, let $*$ be a base point, and let $G=\pi_1(M,*)$ be the fundamental group of the surface $M$. The boundary map in the group homology is defined by 
		\begin{equation*}
			\begin{aligned}
				\partial :G\times G\rightarrow\mathbb ZG&&(X,Y)\mapsto X+Y-XY.
			\end{aligned}
		\end{equation*}
For any bounded group element $A\in G$, we want to construct a multi-curve $\mathcal R_G(A)$. And for any bounded triangle $(X,Y)\in G\times G$, we want to construct a multi-pants $\mathcal R_{G\times G}(X,Y)$ such that the following diagram commutes.
\[\begin{tikzcd}
		G\times G\arrow[d,"\partial"] \arrow[r,"\mathcal{R}_{G\times G}"]& \frac{\mathbb Z\PPi_{\epsilon,R}}{N}\arrow[d,"\partial"]\\
		G\arrow[r,"\mathcal R_G"]&\frac{\mathbb Z\Gamma_{\epsilon,R}}{N}
	\end{tikzcd}\]

	\subsection{Narrow triangle and some preliminaries}
	This and the subsequent subsections are devoted to prove the Narrow Triangle Replacement Lemma.\par  
	We first introduce the notion of narrow triangles. 
	\begin{definition}[Narrow triangle]
		Let $u$ be a unit vector, and let $\epsilon,R>0$. An element $(A_1,A_2)\in LN_{\epsilon,R}(u)\times LN_{\epsilon,R}(u)$ is called a narrow triangle if 
		\begin{enumerate}
			\item $A_3=[\cdot\bar A_2\bar A_1\cdot]\in LN_{\epsilon,R}(u).$
			\item $\Delta(A_1,A_2)\leq  1.5R$, where
			\begin{equation*}
		\begin{aligned}
			\Delta(A_1,A_2)=\max\bigg\{&\len(A_1)+\len(A_2)-\len(A_3)\\
			&,\len(A_1)+\len(A_3)-\len(A_2),\len(A_2)+\len(A_3)-\len(A_1)\bigg\}.
		\end{aligned}
	\end{equation*}
		\end{enumerate}
		Let $(LN\times LN)_{\epsilon,R}(u)$ be the set consisting of all narrow triangles.
	\end{definition}

 In the reminder of this subsection, we prove some preliminaries for the subsequent subsection.

	\begin{proposition}
		Let $(A_1,A_2)$ be a narrow triangle. There exists a unique triple of geodesic arcs $(P_1,P_2,P_3)$ with the same initial point and the same terminal point such that
		\begin{enumerate}
			\item $i(P_1)=-i(P_2)$ and $i(P_3)\perp i(P_1)$.
			\item $A_1=[\cdot\bar P_2P_3\cdot],A_2=[\cdot\bar P_3P_1\cdot]$ and $A_3=[\cdot\bar P_1P_2\cdot]$.
		\end{enumerate}
		The union of $P_1,P_2,P_3$ is called the $T$-graph of the narrow triangle.
	\end{proposition}
	\begin{proof} 
		For any two points $p,q\in\mathbb H$, by $[pq]$ we denote the oriented geodesic arc from the point $p$ to the point $q$. The consecutive geodesic arcs $A_1,A_2,A_3$ can be lifted to the hyperbolic plane $\mathbb H$ as a piecewise geodesic curve, which bounds a hyperbolic triangle in $\mathbb H$. By $a_1,a_2,a_3$ we denote three vertex of the triangle such that $[a_1a_2],$ $[a_2a_3],[a_3a_1]$ are the lifting of $A_3,A_1,A_2$ respectively. Since the triangle is thin, there exists a unique point $p\in[a_1a_2]$ such that $[a_3p]$ is perpendicular to $[a_1a_3]$ at the point $p$. Let $P_i$ be the image of the geodesic arc $[pa_i]$ in $M$ for any $i=1,2,3$. 
	\end{proof}
	
	\begin{definition}
	 	The narrow triangle is said to be of positive orientation if the triple of vectors $(i(P_1),i(P_2),i(P_3))$ has positive cyclic ordering in $T^1M_p\approx S^1$; otherwise the narrow triangle is said to be of negative orientation. Note that for any narrow triangle, $(A_1,A_2)$ and $(\bar A_2,\bar A_1)$ have opposite orientations. 
	 \end{definition}

	\begin{proposition}[Small narrow triangle]\label{small-narrow}
	There exist constants $C>0$ and $\epsilon_1,R_1,q_1>0$ depending only on $M$ such that for any $R\geq R_1$ and $\epsilon_1\geq \epsilon\geq e^{-q_1R}$ the following holds. There exists a narrow triangle $(A_1^0,A_2^0)\in LN_{\epsilon,R}(u)\times LN_{\epsilon,R}(u)$ such that  $\len(A_i^0)\leq -C\log\epsilon$ for any $i=1,2,3$.
	\end{proposition}
	 \begin{proof}
	 	Let $v_0$ be an auxiliary unit vector, and let 
	 	\begin{equation*}
	 		\begin{aligned}
	 		P_1\in& Conn_{\epsilon^3,L(\epsilon^3)}(\sqrt{-1}v_0,u),\\
	 	P_2\in& Conn_{\epsilon^3,L(\epsilon^3)}(-\sqrt{-1}v_0,u),\\
	 	P_3\in& Conn_{\epsilon^3,L(\epsilon^3)}(v_0,u).
	 		\end{aligned}
	 	\end{equation*}
	 	We define $A_1^0=[\cdot\bar P_2P_3\cdot],A_2^0=[\cdot\bar P_3P_1\cdot]$ and $A_3^0=[\cdot\bar P_1P_2\cdot]$.
	 \end{proof}

	\subsection{The Narrow Triangle Replacement Lemma}
	In this subsection, we prove the Narrow Triangle Replacement Lemma by applying the Four-Part Itemization Theorem.
	\begin{lemma}[Narrow Triangle Replacement Lemma]
	There exist constants $\epsilon_1,R_1,q_1>0$ depending only on $M$ such that for any $R\geq R_1$ and $\epsilon_1\geq \epsilon\geq e^{-q_1R}$ the following holds. Suppose $(A_1,A_2)\in (LN\times LN)_{\epsilon,R}(u)$ is a narrow triangle. Then there exists a multi-pants
	\[
	\mathcal{RT}(A_1,A_2)\in \frac{\mathbb Z}{8\lfloor e^{2R}\rfloor^8}\PPi_{\epsilon,R}
	\]
	such that 
	\[
	\partial\mathcal{RT}(A_1,A_2)=\mathcal R(\partial(A_1,A_2)),
	\]
	and the partial map $\mathcal{RT}$ is $(R\epsilon^{-1})^m$-semirandom with respect to the measure classes $\sigma_C^{\times 2}$ and $\sigma_{\PPi}$ for some $m=m(M)>0$.
	\end{lemma}
	\begin{proof}
	
	\textbf{Claim: }Suppose $(A_1,A_2)$ and $(A_1^\prime,A_2^\prime)$ are two narrow triangles with opposite orientations such that $\Delta(A_1,A_2)+\Delta(A_1^\prime,A_2^\prime)\leq 1.6R$. Then there exists a multi-pants
	\[
	Rot(A_1,A_2,A_1^\prime,A_2^\prime)\in\frac{\mathbb Z}{4\lfloor e^{2R}\rfloor^8}\PPi_{\epsilon,R}
	\]
	such that 
	\[
	\partial Rot(A_1,A_2,A_1^\prime,A_2^\prime)=\mathcal R(\partial(A_1,A_2))-\mathcal R(\partial(A_1^\prime,A_2^\prime)).
	\]\par 
	We first construct the map $\mathcal{RT}$ by assuming the claim. Let $(A^0_1,A^0_2)$ be the small narrow triangle in Proposition \ref{small-narrow}. Without lose of generality, we can assume $(A_1,A_2)$ and $(A^0_1,A^0_2)$ have the opposite orientations (otherwise we replace $(A^0_1,A^0_2)$ with $(\bar{A_2^0},\bar{A_1^0})$). Meanwhile, we have 
	\[\Delta(A^0_1,A^0_2)\leq -C\log\epsilon\leq 0.1R\] 
	providing $\epsilon\geq e^{-\frac{R}{10C}}$. Thus, both of two pairs of narrow triangles $(A_1,A_2,A^0_1,A^0_2)$ and $(A^0_1,A^0_2,\bar{A_2^0},\bar{A_1^0})$ satisfy the condition of the claim.  We define
	\[
	\mathcal{RT}(A_1,A_2)=Rot(A_1,A_2,A^0_1,A^0_2)+\frac 12Rot(A^0_1,A^0_2,\bar{A_2^0},\bar{A_1^0}).
	\]
	According to the claim, $\mathcal{RT}(A_1,A_2)$ satisfies the needed property.
	\begin{proof}[\textbf{Proof of the claim}]
	Let $r_1,r_2,r_3$ be constants defined by the following equations
		\begin{equation*}
			\begin{aligned}
				r_1+r_2&=2R-\len(A_3)-\len(A_3^\prime)\\
				r_2+r_3&=2R-\len(A_1)-\len(A_1^\prime)\\
				r_1+r_3&=2R-\len(A_2)-\len(A_2^\prime)
			\end{aligned}
		\end{equation*}
		Let $\mathcal F_i=Conn_{\epsilon^2,r_i}(u,-u)$. Since $\Delta(A_1,A_2)+\Delta(A_1^\prime,A_2^\prime)\leq 1.6R$, we have $r_i\geq 0.2R$ for any $i=1,2,3$.\par 
		Let $(R_1,R_2,R_3)$ be the $T$-graph of $(A_1,A_2)$, and let $(S_1,S_2,S_3)$ be the $T$-graph of $(A_1^\prime,A_2^\prime)$. Since two $T$-graphs have opposite orientations, for any triple of elements $B=\{B_i\in\mathcal F_i\}_{i=1,2,3}$, the union of the triple of geodesic arcs 
		\[[\cdot R_1B_1\overline S_1\cdot]\cup[\cdot R_2B_2\overline S_2\cdot]\cup[\cdot R_3B_3\overline S_3\cdot]\] 
		is a $\theta$-graph. We denote the immersed pair of pants constructed by the $\theta$-graph by $P(A_1,A_2,A_1^\prime,A_2^\prime,B)$. The pair of pants $P$ has the following boundary
		\begin{equation*}
			\begin{aligned}
				\partial P(A_1,A_2,A_1^\prime,A_2^\prime,B)=\left[A_3B_2\bar A_3^\prime \bar B_1\right]+\left[A_1B_3\bar A_1^\prime\bar B_2\right]+\left[A_2B_3\bar A_2^\prime\bar B_1\right].
			\end{aligned}
		\end{equation*}
		By applying the Chain Lemma, the above three geodesic curves are good curves. Therefore, $P\in\PPi_{\epsilon,R}$. Let
		\begin{equation*}
			\begin{aligned}
		Rot^*(A_1,A_2,A_1^\prime,A_2^\prime,B)&=P(A_1,A_2,A_1^\prime,A_2^\prime,B)-Item(A_3,B_2,\bar A_3^\prime, \bar B_1)\\
				-Item&(A_1,B_3,\bar A_1^\prime,\bar B_2)-Item(A_2,B_3,\bar A_2^\prime,\bar B_1),
			\end{aligned}
		\end{equation*}
		By applying the Four-Part Itemization Theorem, for any  $B\in\mathcal F_1\times\mathcal F_2\times\mathcal F_3$, the following holds
		\begin{equation*}
			\begin{aligned}
				\partial Rot^*(A_1,A_2,A_1^\prime,A_2^\prime,B)&=\sum_{i=1}^3\mathcal R(A_i)+\sum_{i=1}^3\mathcal R(\bar{A_i^\prime})\\
				&=\mathcal{R}\big(\partial (A_1,A_2)\big)-\mathcal{R}\big(\partial (A_1^\prime,A_2^\prime)\big).
			\end{aligned}
		\end{equation*}
		Let $\underline B_i$ be an $(\epsilon^2,R)$-effectively random element in $\mathcal F_i$, and let $\underline B=(\underline B_1,\underline B_2,\underline B_3)$. Letting
		\[
		Rot(A_1,A_2,A_1^\prime,A_2^\prime)=Rot^*(A_1,A_2,A_1^\prime,A_2^\prime,\underline B),
		\]
		this completes the proof of the claim.
		\end{proof}
		\textbf{ Estimate the semirandom norm. } 
		We first estimate the semirandom norm of the map $Rot$. By definition, we have
		\begin{equation}
			\begin{aligned}\label{E5}
				|P|_*(\sigma_C^{\times 7})(\Pi)&=\sum_{A_1,A_2,A_1^\prime,A_2^\prime,B}e^{-\len(A_1)-\len(A_2)-\len(A_1^\prime)-\len(A_2^\prime)-\sum_ir_i}
			\end{aligned}
		\end{equation}
		where the summation is over all elements $(A_1,A_2,A_1^\prime,A_2^\prime,B)$ such that $P(A_1,A_2,A_1^\prime,A_2^\prime,B)=\Pi$. In the spirit of Proposition \ref{8}, there are at most $C_1R^8$ many elements $(A_1,A_2,A_1^\prime,A_2^\prime,B)$ in the summation for some $C_1=C_1(M)>0$. According to the definition of $r_i$, we have
		\begin{equation}\label{E6}
			\sum_i r_i\geq  3R-\len(X)-\len(Y)-\len(Z)-\len(W).
		\end{equation}
		Combining this with estimate (\ref{E5}), we have $||P||_{s.r.}\leq C_1R^8$. Since $||Rot^*||_{s.r.}$ is bounded by the semirandom norms of the map $P$ and the map $Item$, we have $||Rot^*||_{s.r.}\leq (R\epsilon^{-1})^{30}$. The randomized map $Rot$ is a $(R\epsilon^{-1})^{40}$-semirandom map.\par 
		The map $\mathcal{RT}$ is the composition of the map $Rot$ and the map 
		\[
		(A_1,A_2)\mapsto (A_1,A_2,A^0_1,A^0_2)-\frac 12(A^0_1,A^0_2,\bar{A_2^0},\bar{A_1^0}).
		\]
		Since $(A_1^0,A_2^0)$ is a small narrow triangle and $\len(A_i^0)\leq -C\log\epsilon$ for any $i=1,2$, we obtain that the latter map is $\epsilon^{-4C}$-semirandom, where $C=C(M)$ is the constant in Proposition \ref{small-narrow}. This implies that $||\mathcal{RT}||_{s.r.}\leq (R\epsilon^{-1})^m$ for some $m=m(M)>0$, and completes the proof.
	\end{proof}

	\subsection{Bounded group element and bounded triangle}
 In this subsection, given $\epsilon,R>0$ and a unit vector $v$, we define the notion of bounded group element and bounded triangle.\par 
 Let $*$ be a base point, and let $G=\pi_1(M,*)$ be the fundamental group of the surface $M$. A group element $A\in G$ is identified with the geodesic arc from $*$ to itself in the homotopy class of $A$. Let $v\in T^1M_*$ be a unit vector supported on $*$. For any $l>0$, by $e^{-lv}$ we denote the length $l$ geodesic arc with initial vector $-v$. For any $A\in G$, we denote
  	\[A^{lv}=[\cdot\overline{e^{-lv}}Ae^{-lv}\cdot].\]
	  \begin{definition}[Stretching maps]
	For any $A\in G$, we define 
	\[Str(A):=A^{L(\epsilon^3)v},\] 
	 where 
	 \[L(\epsilon^3)=\max\{-q_0^{-1}\log\epsilon^3,-\log\epsilon^3,L_0\}\]
	  is the constant in the Connection Lemma. For any $(X,Y)\in G\times G$, we define $Str(X,Y)=(Str(X),Str(Y))$.
	\end{definition}
	From now on, by $u$ we denote the terminal unit vector of the geodesic arc $e^{-L(\epsilon^3)v}$. We define the set of bounded group elements as the pre-image of the left narrow connections with respect to the vector $u$.
	\begin{definition}
		A group element $A\in G$ is called $(\epsilon,R,v)$-bounded if $Str(A)\in LN_{\epsilon,R}(u)$. An element $(X,Y)\in G\times G$ is called a $(\epsilon,R,v)$-bounded triangle if $Str(X,Y)\in (LN\times LN)_{\epsilon,R}(u)$.
	\end{definition} 
	The length of the stretchment is carefully chosen, such that the semirandom norm of the stretching map is not too large (see Proposition \ref{T-semi}), and the geometry of the stretched element $Str(A)$ can be effectively controlled by the inefficiency of $A$ (see Lemma \ref{T-esit}).
	\begin{proposition}\label{T-semi}
		Let $q_0$ be the constant in Connecting Principle. There exist a constant $\epsilon_1>0$ depending only on $M$ such that for any $\epsilon_1\geq \epsilon>0$ the stretching map $Str$
		is $\epsilon^{-m}$-semirandom with respect to the measure classes $\sigma_G$ and $\sigma_C$ for $m=6(q_0^{-1}+1)$. 
	\end{proposition}
	\begin{proof}
	Because the stretching map $Str$ is injective, the following holds
		\begin{equation*}
			\begin{aligned}
				|Str|_*(\sigma_G)(Str(A))=\sigma_G(A)=\frac{e^{\len(Str(A))}}{e^{\len(A)}}\sigma_C(Str(A)).
			\end{aligned}
		\end{equation*}
		By applying the triangle inequality, we obtain that
		\[
		\len(Str(A))\leq \len(A)+2L(\epsilon^3).
		\]
		Thus, letting $\epsilon_1=e^{-L_0/3}$, we get $||Str||_{s.r.}\leq e^{2L(\epsilon^3)}\leq\epsilon^{-m}$.
	\end{proof}

	For any group element $A\in G$, we define the inefficiency of $A$ with respect to $v$ as 
  	\[
I_v(A)=\lim_{l\rightarrow\infty}I\left(\overline{e^{-lv}}Ae^{-lv}\right).
\]
The following proposition establishes that the group elements with bounded inefficiency is uniformly bounded.
	\begin{lemma}\label{T-esit}
	There exist constants $\epsilon_1,R_1,q_1>0$ depending only on the surface $M$ and the constant $D>0$ such that for any $R\geq R_1$ and $\epsilon_1\geq \epsilon\geq e^{-q_1R}$ the following holds. Suppose $A\in G$ such that $\len(A)\leq 1.5R$ and $I_v(A)\leq D$. Then $A$ is a $(\epsilon,R,v)$-bounded group element.\par  Moreover, suppose $(X,Y)\in G\times G$ such that $X,Y,XY\in G$ are $(\epsilon,R,v)$-bounded and $\Delta(X,Y)\leq 1.4R$, then $(X,Y)$ is a $(\epsilon,R,v)$-bounded triangle.
	\end{lemma}
	\begin{proof}
	Since for any $l>0$
  	\[I\left(\overline{e^{-lv}}Ae^{-lv}\right)\leq I_v(A)\leq D,\]
  	by the definition of the inefficiency, we have
  	\[
  	\len(A)+2l-\len(A^{lv})=I\left(\overline{e^{-lv}}Ae^{-lv}\right)\leq D.
  	\] 
  	Combining this with the triangle inequality, for any $l>0$, we get
  	\[
  	\len(A)+2l-D\leq\len(A^{lv})\leq \len(A)+2l.
  	\]
  	In particular, we have the following length estimate
  	\begin{equation}\label{Eq1}
  		2L(\epsilon^3)-D\leq\len(Str(A))\leq 1.5R+2L(\epsilon^3).
  	\end{equation}
	By applying Lemma \ref{new-angle}, the following angle estimate holds for any $l>D+1$,
  		\begin{equation*}
		\begin{aligned}
			\Theta\bigg(i\left(A^{lv}\right),t\left(\overline{e^{-lv}}\right)\bigg),\Theta\bigg(t\left(A^{lv}\right),t\left(e^{-lv}\right)\bigg)&\leq Ce^{-l},
		\end{aligned}		
	\end{equation*}
	for some constant $C=C(D)>0$. In particular, we obtain  
	\begin{equation}\label{Eq2}
		Str(A)\in Conn_{C\epsilon^3,<\infty}(-u,u).
	\end{equation}
	Let $\epsilon_1,R_1,q_1>0$ be constnats such that for any $R\geq R_1$ and $\epsilon_1\geq \epsilon\geq e^{-q_1R}$ we have
	\[
	L(\epsilon^2)\leq 2L(\epsilon^3)-D,
	\]
	\[
	1.5R+2L(\epsilon^3)\leq 2R-L(\epsilon^2),
	\]
	and 
	\[
	C\epsilon^3\leq\epsilon^2.
	\]
	Combining these assumptions with Equations (\ref{Eq1}) and (\ref{Eq2}), we obtain that $Str(A)\in LN_{\epsilon,R}(u)$. This implies that $A$ is $(\epsilon,R,v)$-bounded. In order to prove that the moreover part, it suffices to prove that $\Delta(Str(X,Y))\leq 1.5R$. This is establishes by the following estimate
	\[
	\Delta(Str(X),Str(Y))\leq \Delta(X,Y)+6L(\epsilon^3).
	\]
	\end{proof}

	\subsection{The bounded group homology and the Group-to-Pants Theorem}
	Recall that the boundary map in the group homology is defined by 
		\begin{equation*}
			\begin{aligned}
				\partial :G\times G\rightarrow\mathbb ZG&&(X,Y)\mapsto X+Y-XY.
			\end{aligned}
		\end{equation*}
 By $G_{\epsilon,R}(v)$ we denote the set of all $(\epsilon,R,v)$-bounded group elements. By $(G\times G)_{\epsilon,R}(v)$ we denote the set of all $(\epsilon,R,v)$-bounded triangles. The restriction of the standard group homology is called the bounded group homology
 \[
 \partial:(G\times G)_{\epsilon,R}(v)\rightarrow\mathbb ZG_{\epsilon,R}(v).
 \] 
 The main purpose of this section is to prove the following theorem, which establishes the connection between the bounded group homology and the good pants homology. So that the theorem is called the Group-to-Pants Theorem.
\begin{theorem}[Group-to-Pants Theorem]
		There exist constants $R_1,\epsilon_1,q_1>0$ depending only on the surface $M$ such that for any $R\geq R_1$ and $\epsilon_1\geq \epsilon\geq e^{-q_1R}$ the following holds. Let $v$ be any unit vector in $T^1M_*$. There exist maps
		\[
		\mathcal{R}_G:G_{\epsilon,R}(v)\rightarrow \frac{\mathbb Z\Gamma_{\epsilon,R}}{8\lfloor e^{2R}\rfloor^8},
		\]
		and 
		\[\mathcal{R}_{G\times G}:(G\times G)_{\epsilon,R}(v)\rightarrow \frac{\mathbb Z\PPi_{\epsilon,R}}{8\lfloor e^{2R}\rfloor^8},\]
		such that for any $(X,Y)\in (G\times G)_{\epsilon,R}(v)$ we have
		\[
		\partial \mathcal R_{G\times G}(X,Y)=\mathcal R_G\big(\partial (X,Y)\big),
		\]
	and the map $\mathcal {R}_{G\times G}$ is $(R\epsilon^{-1})^m$-semirandom with respect to the measure classes $\sigma_G\times\sigma_G$ and $\sigma_{\PPi}$  for some $m=m(M)>0$.
	\end{theorem}
	\begin{proof}
		
	 	 According to the definition of the bounded group homology,  the left square in the following diagram naturally commutes.
	 \[
 \begin{tikzcd}
		(G\times G)_{\epsilon,R}(v)\arrow[d,"\partial"] \arrow[r,"Str"]&(LN\times LN)_{\epsilon,R}(u)\arrow[d,"\partial"] \arrow[r,"\mathcal{RT}"]&\frac{\mathbb Z\PPi_{\epsilon,R}}{8\lfloor e^{2R}\rfloor^8}\arrow[d,"\partial"]\\
		G_{\epsilon,R}(v)\arrow[r,"Str"]&LN_{\epsilon,R}(u) \arrow[r,"\mathcal{R}"]&\frac{\mathbb Z\Gamma_{\epsilon,R}}{8\lfloor e^{2R}\rfloor^8}
	\end{tikzcd}
 \]
 According to the Narrow Triangle Replacement Lemma, there exists a map $\mathcal{RT}$ such that the right square in the above diagram commutes.
  We define
	 	\[\mathcal R_G=\mathcal R\circ Str,\]
	and
	 	 \[\mathcal R_{G\times G}=\mathcal{RT}\circ Str.\]
	 We obtain that the map $\mathcal R_{G\times G}$ satisfies the needed property
	 	\[
	 	\partial\mathcal R_{G\times G}(X,Y)=\mathcal R\big(Str\text{ }\partial (X,Y)\big)=\mathcal R_G\big(\partial (X,Y)\big).
	 	\]	 
	 	The semirandom norm of the map $\mathcal{R}_{G\times G}$ is bounded by the semirandom norm of the map $Str$ and the map $\mathcal{RT}$. Thus, by the Narrow Triangle Replacement Lemma and Proposision \ref{T-semi}, the map $\mathcal {R}_{G\times G}$ is $(R\epsilon^{-1})^m$-semirandom  for some $m=m(M)>0$. This completes the proof.

	\end{proof}

\section{Proof of the Good Pants Homology Theorem}
In this section, we prove the Good Pants Homology Theorem. In order to construct the correction map $\Phi$, for any good curve $\gamma$ we need to construct the multi-pants $\Phi(\gamma)$. The idea of the construction is first encoding the good curve $\gamma$ in group elements, then applying the Group-to-Pants Theorem.

	\subsection{Dichotomy of good curves}
	In this subsection, by applying dichotomy argument, we encode any good curve $\gamma$ in two bounded group elements $X,Y\in G$ such that $\gamma=[XY]$. The dichotomy depends on the choice of a unit vector $v\in T^1M_*$, which will be fixed in the Bounded Group Homology Lemma.
	\begin{definition}
		Given $\epsilon,R>0$ and a unit vector $v\in T^1M_*$. For any good curve $\gamma\in\Gamma_{\epsilon,R}$. Let $a,b\in\gamma$ be two antipodal points. By $\gamma_{ab}$ we denote the subarc of $\gamma$ from $a$ to $b$. By $\gamma_{ba}$ we denote the subarc of $\gamma$ from $b$ to $a$. We take a pair of geodesic arcs
		\[
		S_a\in Conn_{\epsilon^3,\frac{R}{2}-\frac{\len(\gamma)}{4}+L(\epsilon^3)+\log 2}\left(v,\sqrt{-1}\gamma^\prime(a)\right)
		\]
		\[
		S_b\in Conn_{\epsilon^3,\frac{R}{2}-\frac{\len(\gamma)}{4}+L(\epsilon^3)+\log 2}\left(v,\sqrt{-1}\gamma^\prime(b)\right),
		\]
		where $L(\epsilon^3)$ is the constant in the Connection Lemma. We define 
		\begin{equation*}
			\begin{aligned}
				X(\gamma)=\left[\cdot S_a\gamma_{ab}\overline S_b\cdot\right]&\text{, and }&Y(\gamma)=\left[\cdot S_b\gamma_{ba}\overline S_a\cdot\right].
			\end{aligned}
		\end{equation*}
		to be the dichotomy of the curve $\gamma$. Its obvious that $[XY]=\gamma$.
	\end{definition}
	\begin{remark}
		Since
		\[\frac{R}{2}-\frac{\len(\gamma)}{4}+L(\epsilon^3)+\log 2\approx L(\epsilon^3)+\log 2,\] 
		the above two connection sets are nonempty by the Connection Lemma. This promises the existence of $S_a,S_b$. 
	\end{remark}
	\begin{proposition}\label{dicho-curve}
		There exist constants $\epsilon_1,R_1>0$ depending only on $M$ such that for any $R\geq R_1$ and $\epsilon_1\geq\epsilon>0$ the following holds. For any $\gamma\in\Gamma_{\epsilon,R}$, we have
		\[X,Y\in Conn_{\epsilon^2,R+2L(\epsilon^3)}(v,-v).\] 
	\end{proposition}
	\begin{proof}
	By applying the Right Angle Chain Lemma to the consecutive geodesic arcs $S_a,\gamma_{ab},\overline S_b$, the following estimates hold
		\begin{equation*}
			\begin{aligned}
				|\len(X)-R-2L(\epsilon^3)|\leq C\epsilon^3\\
				\Theta(i(X),v),\Theta(t(X),-v)\leq C\epsilon^3
			\end{aligned}
		\end{equation*}
		for some universal constant $C$. The same estimate holds for $Y$. Let $\epsilon_1=C^{-1}$. Thus, we have
		\[X,Y\in Conn_{\epsilon^2,R+2L(\epsilon^3)}(v,-v).\] 
	\end{proof}
	Thus, according to Proposition \ref{dicho-curve} and Lemma \ref{bounded-ineff}, the dichotomy elements $X(\gamma)$ and $Y(\gamma)$ have bounded inefficiency, namely,
	 \[
	 I_v(X(\gamma)),I_v(Y(\gamma))\leq D_1,
	 \] 
	 where $D_1$ is the  universal constant in Lemma \ref{bounded-ineff}. Therefore, by Lemma \ref{T-esit}, the dichotomy elements are $(\epsilon,R,v)$-bounded group elements.\par
	 The dichotomy construction defines two maps
	 \[X:\Gamma_{\epsilon,R}\rightarrow G_{\epsilon,R}(v),\]
	 and
	 \[
	 Y:\Gamma_{\epsilon,R}\rightarrow G_{\epsilon,R}(v).
	 \] 
	 The following proposition estimates the semirandom norm of $X,Y$.
	\begin{proposition}
		There exist constants $\epsilon_1,R_1>0$ depending only on $M$ such that for any $R\geq R_1$ and $\epsilon_1\geq\epsilon>0$ the following holds. The maps $\gamma\mapsto X(\gamma)$ and $\gamma\mapsto Y(\gamma)$ are $(R\epsilon^{-1})^m$-semirandom with respect to the measure classes $\sigma_\Gamma$ and $\sigma_G$ for some constant $m=m(M)>0$.
	\end{proposition}
	\begin{proof}
	We first prove that the map 
	\[\gamma\mapsto \big(X(\gamma),Y(\gamma)\big)\] is $(R\epsilon^{-1})^n$-semirandom. The following estimates holds
			\begin{equation*}
			\begin{aligned}
				|(X,Y)|_*\sigma_\Gamma(X,Y)=\len(\gamma)e^{-\len(\gamma)}&\leq \len(\gamma)e^{-\len(X)-\len(Y)+2\len(S_a)+2\len(S_b)}\\&
				=\len(\gamma)e^{2\len(S_a)+2\len(S_b)}\sigma_G^{\times 2}(X,Y)\\
				&\leq CRe^{4L(\epsilon^3)}\sigma_G^{\times 2}(X,Y)
			\end{aligned}
		\end{equation*}
		where $C$ is an universal constant. The first inequality in the above estimate comes from the triangle inequality and the second inequality comes from the length estimates about $\gamma,S_a,S_b$.  Let $n$ be sufficiently large such that $CRe^{2L(\epsilon^3)}\leq (R\epsilon^{-1})^n$. According to Proposition \ref{forgetful}, the forgetful partial map 
		\begin{equation*}
			\begin{aligned}
				\big(X(\gamma),Y(\gamma)\big)\mapsto X(\gamma)
			\end{aligned}
		\end{equation*}
		 is $C^\prime R$-semirandom for some constatn $C^\prime=C^\prime(M)>0$. This implies that the map $\gamma\mapsto X(\gamma)$ is $(R\epsilon^{-1})^{n+2}$-semirandom when $R_1\geq C^\prime$. The same estimates holds for $\gamma\mapsto Y(\gamma)$. This completes the proof.
	\end{proof}
	\subsection{The Curve-to-Group Lemma}
	In this subsection, we prove the Curve-to-Group Lemma. 
	
	\begin{lemma}[Curve-to-Group]
	There exist constants $\epsilon_1,R_1,q_1>0$ depending only on $M$ such that for any $R\geq R_1$ and $\epsilon_1\geq\epsilon\geq e^{-q_1R}$ the following holds. For any $\gamma\in\Gamma_{\epsilon,R}$ there exists a multi-pants
		\[\mathcal{RC}(\gamma)\in \frac{\mathbb Z}{2\lfloor e^{2R}\rfloor^6}\PPi_{\epsilon,R} \] such that 
			\[
			\partial  \mathcal{RC}(\gamma)=\mathcal{R}_G\big(X(\gamma)\big)+\mathcal{R}_G\big(Y(\gamma)\big)-\gamma,
			\]
			and the map $\mathcal{RC}$ is $(Re^{-1})^m$-semirandom with respect to the measure classes $\sigma_{\Gamma}$ and $\sigma_{\PPi}$ for some constant $m=m(M)>0$.
	\end{lemma}
	\begin{proof}
	 Let $\mathcal F=Conn_{\epsilon^2,R-2L(\epsilon^3)-2\len(T)}(u,-u)$ be a subset of the right narrow connection set such that for any $B\in\mathcal F$, the union of geodesic arcs 
	\[[\cdot \overline S_aTB\overline TS_b\cdot ]\cup\gamma_{ab}\cup\gamma_{ba}\]
	 is a $\theta$-graph. By $P(\gamma,B)$ we denote the pair of pants constructed by this graph. Thus, we get
	\begin{equation}\label{E3}
		\partial P(\gamma,B)=-\gamma+[Str(X)B]+[Str(Y)\bar B].
	\end{equation}
	By applying the Two-Part Itemization Theorem, we get
	\begin{equation}\label{E4}
		\begin{aligned}
			\text{ }[Str(X)B]&=\partial Item_2(Str(X),B)+\mathcal R(Str(X))+\mathcal R(B),\\
			\text{ }[Str(Y)\bar B]&=\partial Item_2(Str(Y),\bar B)+\mathcal R(Str(Y))+\mathcal R(\bar B).
		\end{aligned}
	\end{equation}
	Letting   
	\[
	\mathcal{RC}^*(\gamma,B)=P\left(\gamma,B\right)-Item_2(Str(X),B)-Item_2(Str(Y),\bar B).
	\]
	for any $B\in\mathcal F$, by taking the summation of the equations (\ref{E3}) and (\ref{E4}), we get
	\begin{equation*}
		\begin{aligned}
			\partial \mathcal{RC}^*(\gamma,B)&=-\gamma+\mathcal{R}(Str(X))+\mathcal{R}(Str(Y))\\
			&=-\gamma+\mathcal{R}_G(X)+\mathcal{R}_G(Y).
		\end{aligned}
	\end{equation*} 
	Let $\underline B$ be an $(\epsilon^2,R)$-effectively random element in $\mathcal F$. We define
	\[\mathcal{RC}(\gamma)=\mathcal{RC}^*(\gamma,\underline B).\] The multi-pants $\mathcal{RC}(\gamma)$ satisfies the conditions from the lemma.
	\par\textbf{Semirandom norm of $\mathcal{RC}$.} 
	The map $P$ is $Re^{4L(\epsilon^3)}$-semirandom with respect to the measure classes $\sigma_\Gamma\times\sigma_C$ and $\sigma_{\PPi}$, because of the following estimates
	\begin{equation*}
		\begin{aligned}
			|P|_*\left(\sigma_\Gamma\times\sigma_C\right)&\big(P(\gamma,B)\big)=\sigma_\Gamma(\gamma)\sigma_C(B)\\
			&=\frac{R}{e^{2R}}\frac{1}{e^{R-4L(\epsilon^3)}}=Re^{4L(\epsilon^3)}\sigma_{\PPi}\big(P(\gamma,B)\big),
		\end{aligned}
	\end{equation*} 
	where the first equality comes from the fact that the map $P$ is injective. The semirandom norm of $\mathcal{RC}^*$ is bounded by the semirandom norm of the map $P$, the map $Item_2$, the stretch map $Str$, and the maps $X$ and $Y$. Since the semirandom norms of these maps are the polynomials of $R\epsilon^{-1}$, we get that the map $\mathcal{RC}^*$ is $(R\epsilon^{-1})^n$-semirandom for some $n=n(M)>0$. Therefore, the randomized map $\mathcal{RC}$ is $(R\epsilon^{-1})^m$-semirandom for some $m=m(M)$.
	\end{proof}
	\subsection{Dichotomy of group elements}
	In this subsection, by applying dichotomy argument, we divide a group element $A$ into two group elements $X,Y\in G$ such that $A=XY$. The dichotomy depends on the choice of a unit vector $v\in T^1M_*$, which will be fixed in the Bounded Group Homology Lemma.
	\begin{definition}
		Let $v\in T^1M_*$ be a unit vector, and let $K>0$ be a constant such that the connection set $Conn_{1,K}(v_1,v_2)$ is nonempty for any $v_1,v_2\in T^1M$. For any $A\in G$, we fix a geodesic arc
		\[L_A\in Conn_{1,K}\left(v,\sqrt{-1}A^\prime\big(\len(A)/2\big)\right),\]
		and define the dichotomy elements of $A$ by
		\begin{equation*}
			\begin{aligned}
				X(A)=[\cdot A|_{[0,\len(A)/2]}\bar L_A\cdot ]&\text{, and }&Y(A)=[\cdot L_AA|_{[\len(A)/2,\len(A)]}\cdot].
			\end{aligned}
		\end{equation*}
		By $\mathcal D^i(A)$ we denote the set consisting of $2^i$ elements occurring in the $i^{th}$  dichotomy of the element $A$. 
	\end{definition}

	The following lemma fix the constant $K$. From now on, we consider $K$ as a constant depending only on the surface $M$. Therefore, the dichotomy of group elements depends only on the choice of the unit vector $v\in T^1M_*$.
	\begin{lemma}\label{dicho-connec}
		 There exist constants $\epsilon_1,R_1,q_1>0$ and $K>0$ depending only on $M$ such that for any $R\geq R_1$ and $\epsilon_1\geq\epsilon\geq e^{-q_1R}$ the following holds. For any group element $A\in G$,  we denote $m(A)=\lfloor\log_2\len(A)\rfloor-1$. Suppose $A\in Conn_{1,1.25R}(v,-v)$. Then 
		\begin{enumerate}
			\item for any $B\in\mathcal D^{m(A)}(A)$, we have $\len(B)\leq 3K$.
			\item for any $i<m(A)$ and $B\in\mathcal D^{i}(A)$, we have 
			\[\big(X(B),Y(B)\big)\in (G\times G)_{\epsilon,R}(v).\]
		\end{enumerate}
	\end{lemma}
	\begin{proof}
		For any group element $A\in G$,  we denote 
		\[
		\Theta(A)=\max\left\{\Theta\big(i(A),v\big),\Theta\big(t(A),-v\big)\right\}.
		\]According to hyperbolic trigonometry, for any $A\in G$ the following estimates hold
		\begin{equation}\label{E1.4}
			\left|\len\big(X(A)\big)-\frac{\len(A)}{2}-K\right|\leq C,
		\end{equation}
		
		\[
		\Theta\big(X(A)\big)\leq \max\left\{1+Ce^{-K},\Theta(A)+Ce^{-\frac{\len(A)}{2}}\right\}
		\]
		for some universal constant $C$. The same estimates holds for $Y(A)$. By induction, for any $B\in\mathcal D^i(A)$ we have the following length estimate
		\begin{equation}\label{E1.5}
			\begin{aligned}
				&\left|\len(B)-\left(\frac{\len(A)}{2^i}+\left(2-\frac{1}{2^{i-1}}\right)K\right)\right|\leq C_1	
			\end{aligned}
		\end{equation} 
		for some universal constant $C_1$ and the following angle estimate
		\begin{equation}\label{E1.6}
		\begin{aligned}
			\Theta\big(B\big)&\leq \max\{1,\Theta(A)\}+Ce^{-K}+Ce^{-C_1}\left(\sum_{s=1}^ie^{-\frac{\len(A)}{2^s}-\left(2-\frac{1}{2^{s-1}}\right)K}\right)\\
				&\leq 1+C_2e^{-K}\left(1+\frac{e^{-\len(A)/2^i}}{1-e^{-\len(A)/2^i}}\right)
		\end{aligned}
		\end{equation}
		for some universal constant $C_2$. By applying Estimate (\ref{E1.5}), for any $B\in\mathcal D^{m(A)}(A)$, we have $\len(B)\leq 2K+C_1+1$. Letting $K\geq C_1+1$, we get the first property. For any $i\leq m(A)$ we have
		\[
		\frac{e^{-\len(A)/2^i}}{1-e^{-\len(A)/2^i}}\leq \frac{1}{1-e^{-\frac 12}}.
		\]
		Combining this with Estimate (\ref{E1.6}), we obtain that there exists a universal constant $C_3>0$ such that, for any $K>C_3$ and any $B\in\mathcal D^i(A), i\leq m(A)$, we have $\Theta(B,v)\leq 1.1$. By Lemma \ref{bounded-ineff}, there exists a constant $D_1>0$ such that 
		\[I_v(B)\leq D_1\] for any $B\in\mathcal D^i(A)$ and $i\leq m(A)$. The estimate (\ref{E1.5}) also implies 
		\[\Delta(X(A),Y(A))\leq \len(A)+2K+2C_1\leq 1.4R.\]
		when $R_1\geq 40K$. By applying Lemma \ref{T-esit}, we obtain that $\big(X(B),Y(B)\big)\in (G\times G)_{\epsilon,R}(v)$. This completes the proof.
		\end{proof}
		
	\subsection{The Bounded Group Homology Lemma}
	In this subsection, we prove the Bounded Group Homology Lemma.\par 
	Let $\{h_1,\dots,h_{2g}\}\subset G$ be a set of generators of the surface group $G$. Thus, we can talk about the word length of the group element. We identify the homology group $H_1(M,\mathbb Z)$ with $\mathbb Z\{h_1,\dots,h_{2g}\}\subset\mathbb ZG$, and define
 	 \[
 	 H:G\rightarrow H_1(M,\mathbb Z)\subset\mathbb ZG
 	 \]
 	 by setting $H(A)$ equal to the homology class of $A$.
 	 \begin{lemma}[Bounded Group Homology Lemma]
 	 There exists a unit vector $v\in T^1M_*$ and constants $\epsilon_1,R_1,q_1>0$ depending only on $M$ such that for any $R\geq R_1$ and $\epsilon_1\geq\epsilon\geq e^{-q_1R}$ the following holds. Suppose   $A\in Conn_{1,\leq 1.25R}(v,-v)$. Then there exists an element $\phi(A)\in\mathbb Z(G\times G)_{\epsilon,R}(v)$ such that
		\[
		\partial \phi(A)=A-H(A),
		\]
		and the partial map $\phi$ is a $R^{n}$-semirandom map for some constant $n=n(M)>0$ with respect to the measure classes $\sigma_G$ and $\sigma_G\times \sigma_G$.

	\end{lemma}
	\begin{definition}
			By $D:G\rightarrow \mathbb ZG$ we denote the map defined by $D(A)=X(A)+Y(A)$. Thus $D^i(A)$ is the summation of all elements in $\mathcal D^i(A)$. By $(X,Y):G\rightarrow G\times G$ we denote the map $(X,Y)(A)=(X(A),Y(A))$.  For any group element $A\in G$, let $m(A)=\lfloor\log_2\len(A)\rfloor-1$ and define
		\[
		\phi_1(A)=-\sum_{i=0}^{m-1}(X,Y)D^i(A).
		\]
		Thus, $\partial \phi_1(A)=A-D^{m(A)}(A)$.
		\end{definition}
		By Lemma \ref{dicho-connec}, the above definition defines a map
		\[
			\phi_1:Conn_{1,\leq 1.25R}(v,-v)\rightarrow  \mathbb Z(G\times G)_{\epsilon,R}(v).
			\]
		\begin{proposition}
			There exists a constant $n=n(M)>0$ such that the map $\phi_1$ is $R^n$-semirandom with respect to the measure classes $\sigma_C$ and $\sigma_G^{\times 2}$.
		\end{proposition}
		\begin{proof}
		By the following computation, the map $(X,Y)$ is $e^{2K}$-semirandom
		\[
		|(X,Y)|_*\sigma_G(X,Y)=e^{-\len(XY)}\leq e^{-\len(X)-\len(Y)+2K}=e^{2K}\sigma_G\times\sigma_G\left(X,Y\right).
		\]
		By the following computation, the map $D$ is $C_3e^{2K}$-semirandom for some constant $C_3$ depends on $M$
		\begin{equation*}
			\begin{aligned}
				|D|_*(\sigma_G)(X)=\sum_{Y}\frac{1}{e^{\len(XY)}}\leq \sum_Y\frac{1}{e^{\len(X)+\len(Y)-2K}}\leq \frac{C_3e^{2K}}{e^{\len(X)}}=C_3e^{2K}\sigma_G(X),
			\end{aligned}
		\end{equation*}
		where the summation over all elements $Y$ satisfies the estimate (\ref{E1.4}). The first inequality in the above estimate comes from the triangle inequality and the second inequality comes from the fact that there are at most $C_3 e^{\len(Y)}$ many $Y$ satisfies the length estimate (\ref{E1.4}). Thus,
		\[
		||\phi_2||_{s.r.}\leq\sum_{i=0}^{\lfloor\log_2 R\rfloor}||(X,Y)D^i||_{s.r.}\leq \sum_{i=1}^{\lfloor\log_2 R\rfloor+1}e^{2K}(C_3e^{2K})^i\leq R^n
		\]
		for some constant $n=n(C_3,K)=n(M)$. 
				\end{proof}
	
	\begin{proof} [Proof of the Bounded Group Homology Lemma]
	 We first fix the unit vector $v$ and the constant $D$. Let $N$ be the minimal natural number such that the word length of any group element $B$ with $\len(B)\leq 3K$ is less than $N$. By $G(N)$ we denote the set of all group elements whose word length is less than $N$. There exists a unit vector $v$ and a constant $D$ such that $I_v(B)\leq D_0$ for any $B\in G(N)$.  Let $D=\max\{D_0,D_1\}$ where $D_1$ is the constant in Lemma \ref{dicho-connec}. By applying Lemma \ref{T-esit}, every $(X,Y)\in G(N)\times G(N)$ is $(\epsilon,R,v)$-bounded.
		\par For any $B\in G(N)$ represented by $B=a_1\cdots a_s$, we define
		\[
		\phi_2(B)=-\sum_{i=1}^{s-1}\left(a_i,a_{i+1}\cdots a_s\right)		\]
		such that the following equation holds
		\[
		\partial \phi_2(B)=a_1\cdots a_s-\sum_ia_i=B-H(B).
		\]
		 Letting $\phi(A)=\phi_1(A)+\phi_2(D^m(A))$, we obtain 
		 \[\partial \phi(A)=A-H(A).\]
		Since $\phi_2$ is only defined on finitely many short elements, the semirandomness of $\phi$ is bounded by $\phi_1$. This completes the proof.
	\end{proof}

	\subsection{Proof of the Good Pants Homology Theorem}
	In this subsection, we prove the Good Pants Homology Theorem. \par 
		 
	 Let $R_1,\epsilon_1,q_1>0$ be constants such that, provided $R\geq R_1$, and $\epsilon_1\geq\epsilon\geq e^{-q_1R}$, the constants $\epsilon,R$ satisfy all assumptions from the Curve-to-Group Lemma, the Group-to-Pants Theorem, and the Bounded Group Homology Lemma. Let $v$ be the unit vector in the Bounded Group Homology Lemma. For any oriented closed curve $\gamma\in\Gamma_{\epsilon,R}$, the needed multi-pants $\Phi(\gamma)$ is constructed as follows:
	 \begin{enumerate}
	 	\item Let $X,Y$ be two dichotomy elements of $\gamma$ constructed in Subsection 7.1. By applying the Curve-to-Group Lemma, we get
	\[
	\gamma=-\partial \mathcal{RC}(\gamma)+\mathcal{R}_G(X)+\mathcal{R}_G(Y).
	\]
	The good curve is encoded in two bounded group elements.
		\item By applying the Bounded Group Homology Lemma, we obtain
	\[
	X=\partial \phi(X)+H(X),
	\] 
	and
	\[
	Y=\partial \phi(Y)+H(Y).
	\]
	Moreover, by applying the Group-to-Pants Theorem, we get
	\[
	\mathcal{R}_G(X)=\partial \mathcal{R}_{G\times G}(\phi (X))+\mathcal{R}_G(H(X))
	\]
	and
	\[
	\mathcal{R}_G(Y)=\partial \mathcal{R}_{G\times G}(\phi( Y))+\mathcal{R}_G(H(Y)).
	\]
	
	\item Let
	\[
	\Phi(\gamma)=-\mathcal{RC}(\gamma)+\mathcal{R}_{G\times G}(\phi (X))+\mathcal{R}_{G\times G}(\phi (Y)),
	\]
	and let $\Psi$ be the restriction of the map $\mathcal R_G$ on the set 
	\[H_1(M,\mathbb Z)=\mathbb Z\{h_1,\dots,h_{2g}\}\subset\mathbb Z G.\]
	Combining the equations in the step 1 and step 2, we get
	\[
	\gamma=\partial \Phi(\gamma)+\mathcal{R}_G(H(X))+\mathcal{R}_G(H(Y)).
	\]
	Since $H(\gamma)=H(X)+H(Y)$,  this implies that the maps $\Phi$ and $\Psi$ satisfy the homology condition.
	 \end{enumerate}
	  The semirandom norm of the map $\Phi$ is bounded by the semirandom norms of the maps $\mathcal{RC},\mathcal R_{G\times G}$ and $\phi$, which are all polynomials of $R\epsilon^{-1}$ depending only on the surface $M$.  This implies the randomness condition and completes the proof.


\begin{thebibliography}{99}  
	
\bibitem{A}L.~V. Ahlfors,  Lectures on quasiconformal mappings, Van Nostrand Mathematical Studies, No. 10, D. Van Nostrand Co., Inc., Toronto, Ont.-New York-London, 1966; MR0200442
\bibitem{KM-surface group}J.~A. Kahn and V. Markovi\'c, Immersing almost geodesic surfaces in a closed hyperbolic three manifold, Ann. of Math. (2) (2012), no.~3, 1127--1190; MR2912704
\bibitem{KM-Correction}J.~A. Kahn and V. Markovi\'c, The good pants homology and the Ehrenpreis conjecture, Ann. of Math. (2) (2015), no.~1, 1--72; MR3374956
\bibitem{Margulis}G.~A. Margulis, Certain applications of ergodic theory to the investigation of manifolds of negative curvature, Funkcional. Anal. i Prilo\v zen.  (1969), no.~4, 89--90; MR0257933
\bibitem{CN}Cayo D\'oria, Nara Paiva, Determining surfaces by short curves and applications,  arXiv:2402.18676v2
\bibitem{Mo}C.~C. Moore, Exponential decay of correlation coefficients for geodesic flows, in  Group representations, ergodic theory, operator algebras, and mathematical physics (Berkeley, Calif., 1984), 163--181, Math. Sci. Res. Inst. Publ., 6, Springer, New York, ; MR0880376
\bibitem{W}S.~A. Wolpert, The length spectrum as moduli for compact Riemann surfaces, in {\it Riemann surfaces and related topics: Proceedings of the 1978 Stony Brook Conference (State Univ. New York, Stony Brook, N.Y., 1978)}, pp. 515--517, Ann. of Math. Stud., No. 97, Princeton Univ. Press, Princeton, NJ, ; MR0624836
\bibitem{EK-dome}Epstein, D. B. A., A. Marden, and V. Markovic. “Quasiconformal Homeomorphisms and the Convex Hull Boundary.” Annals of Mathematics 159, no. 1 (2004): 305–36.
\bibitem{Ari}K. Takeuchi, A characterization of arithmetic Fuchsian groups, J. Math. Soc. Japan {\bf 27} (1975), no.~4, 600--612; MR0398991
\bibitem{Ari-Reid}C. Maclachlan and A.~W. Reid, {\it The arithmetic of hyperbolic 3-manifolds}, Graduate Texts in Mathematics, 219, Springer, New York, 2003; MR1937957

\end{thebibliography}
\end{document}